\documentclass[12pt,a4paper]{amsart}
\usepackage{psfrag}
\usepackage{amsmath, amssymb, graphicx , 
amsthm, 
epsfig, color}
\address[costanti@math.unistra.fr]{Francesco Costantino, Institut de Recherche Math\'ematique Avanc\'ee (IRMA)(CNRS),
Universit\'e de Strasbourg I,
67084 Strasbourg, FRANCE }
\address[murakami@waseda.jp]{Jun Murakami, Department of Mathematics, 
Faculty of Science and Engineering,
Waseda University,
3-4-1 Ohkubo, Shinjuku-ku, 
Tokyo 169-8555, JAPAN}

\numberwithin{equation}{section}

\newtheorem{teo}{Theorem}[section]
\newtheorem{lemma}[teo]{Lemma}
\newtheorem{prop}[teo]{Proposition}

\newtheorem{question}[teo]{Question}

\theoremstyle{definition}
\newtheorem{defi}[teo]{Definition}
\newtheorem{rem}[teo]{Remark}

\theoremstyle{remark}
\newtheorem{prof}[teo]{Proof of}

\newcommand{\qbin}[2]{\left[\begin{array}{c}
#1 \\
#2 \end{array}\right]}

\def\mc{\mathbb{C}}

\def\mz{\mathbb{Z}}

\newcommand{\bZ}{{\mathbf Z}}
\newcommand{\op}{\operatorname{op}}
\newcommand{\col}{\operatorname{col}}
\newcommand{\id}{\operatorname{id}}
\newcommand{\ctg}{\operatorname{ctg}}

\title[$SL(2, {\mathbb C})$ quantum $6j$-symbols and their hyperbolic volume]{On $SL(2, {\mathbb C})$ quantum $6j$-symbol and its relation to the hyperbolic volume}

\author[F. Costantino, J. Murakami]{Francesco Costantino and Jun Murakami}
\thanks{The first author was supported by the French ANR Research Project ANR-08-JCJC-0114-01.
The second author was partially supported by Grant-in-Aid for Scientific Research(C) 19540230. }

\begin{document}

\maketitle

\par
\begin{abstract}
We generalize the colored Alexander invariant of knots to an invariant of graphs, and we construct a face model for this invariant by using the corresponding $6j$-symbols, which come from the non-integral representations of the quantum group ${\mathcal U}_q(sl_2)$.  
We call it the $SL(2, \mathbb C)$ quantum $6j$-symbols, and show their relation to the hyperbolic volume of a truncated tetrahedron.   
\end{abstract}

\begin{bf}
Mathematics Subject Classification (2000). 
\end{bf}
46L37; 46L54, 82B99.  

\section*{Introduction}
The $6j$-symbols were first introduced by Racah for studying atomic spectroscopy and later used by Ponzano and Regge as well as Biedenharn and Louck (and many others) in the study of the theory of gravity by using representation theory of the Lie algebra $sl_2$. Their quantized version first appeared in \cite{KR}, where the face model of the colored Jones invariants of knots and links was constructed using quantum $6j$-symbols instead of quantum $R$-matrices.   
The quantum $6j$-symbols were also used to construct the Turaev-Viro invariant of three manifods (\cite{TV}), which turned out to be the square of the norm of the Witten-Reshetikhin-Turaev invariant (\cite{RT}).  
More recently, R. Kashaev constructed knot invariants from quantized dilogarithm functions and observed (\cite{Kas}) that certain limit of his invariants coincide with the hyperbolic volume of the knot complement; later it turned out (\cite{MM}) that the Kashaev invariant is  the colored Jones invariant of spin $\frac{n-1}{2}$ at $q =\xi_n$, where $\xi_n$ is the primitive $2n$-th root of unity $\exp(\frac{\pi \sqrt{-1}}{n})$. In other words, the Kashaev invariant comes from  the $n$ dimensional irreducible representation of ${\mathcal U}_{\xi_n}(sl_2)$. 
\par
When $q = \xi_n$, there exist other invariants related to  ${\mathcal U}_{\xi_n}(sl_2)$, such as the colored Alexander invariant (\cite{ADO}, \cite{M}, \cite{GR}), the logarithmic invariant (\cite{MN}), and the Hennings invariant (\cite{H}).   
The colored Alexander invariant is related to the central deformation of the $n$-dimensional irreducible representation of ${\mathcal U}_{\xi_n}(sl_2)$,   
which is a non-integral highest weight representation.  
Let $\widetilde{\mathcal U}_{\xi_n}(sl_2)$ be the small (or restricted) quantum group which is a quotient of ${\mathcal U}_{\xi_n}(sl_2)$.
Then the logarithmic invariant is defined by using  the radical part of a non-semisimple representation of $\widetilde{\mathcal U}_{\xi_n}(sl_2)$.
The Hennings invariant is an invariant of 3-manifolds coming from the right integral given by the finite dimensional Hopf algebra structure of $\widetilde{\mathcal U}_{\xi_n}(sl_2)$.  
The logarithmic and Hennings invariants are both related to the logarithmic conformal field theory (\cite{FGST}), and can be expressed in terms of the colored Alexander invariant (\cite{MN}).  
\par
The main purpose  of this paper is to investigate the quantum $6j$-symbols related to the non-integral highest weight representations of ${\mathcal U}_{\xi_n}(sl_2)$, and show their relations to the hyperbolic volume of a truncated tetrahedron.  
In Section 1 after recalling the basic facts about the category of non-integral highest weight representations, we define the Clebsch-Gordan quantum coefficients (CGQC) of the tensor product of two such modules, and then combine them to get the corresponding quantum $6j$-symbols. For $n$ odd, these $6j$-symbols were already given and used to construct a 3-manifold invariants in \cite{GP}, \cite{GPT}. 
Since the spin (parametrizing the irreducible representations of ${\mathcal U}_{\xi_n}(sl_2)$) is a continuous parameter in the above construction, we will call the $6j$-symbols computed through this representation theory the $SL(2, \mathbb C)$-symbols, while we will call the usual quantum $6j$-symbol (for instance those introduced in \cite{KR}) the $SU(2, \mathbb C)$ quantum $6j$-symbols.  For the sake of clarity, we moved all the proofs of algebraic statements in the Appendix.

The second section has the goal of relating the $6j$-symbols to hyperbolic geometry. Many evidences support the idea that geometry should show up while considering asymptotical limits of quantum invariants.
For instance the analysis of such limits for $SU(2, \mathbb C)$ $6j$-symbols, led to the (previously unknown) formulas for the volume of a hyperbolic tetrahedron (\cite{MY}, \cite{U}, \cite{MU}). Similarly various proofs have been provided of special cases of Kashaev's volume conjecture (see for instance \cite{C1}, \cite{C2}, \cite{Kas}, \cite{KT}, \cite{MM}, \cite{Vv}) showing that the hyperbolic volumes of link complements are indeed related to such asymptotical limits, and a geometric explanation for this phenomenon (even though not yet a proof) was provided by Yokota (\cite{Y}).
It was observed in \cite{M}, \cite{CM} that the colored Alexander invariant is related to the hyperbolic volume of a cone manifold whose core is the given knot; this is analogous to the generalized volume conjecture of the Kashaev invariant with deformed $q$ in \cite{G}, \cite{MhY}.    
Here, we show that the $SL(2, \mathbb C)$ quantum $6j$-symbol is related to the volume of a truncated tetrahedron:
\begin{teo}[Theorem \ref{th:trunc} below]
Let $T$ be a truncated hyperbolic tetrahedron with oriented labeled edges, and  
let $0 < \theta_a$, $\theta_b $, $\theta_c$, $\theta_d$, $\theta_e$, $\theta_f < \pi$ be the internal dihedral angles at the edges.  
Let $a_n,b_n,c_n,d_n,e_n,f_n$ be sequences of integers such that $\lim_{n\to \infty} \frac{2\pi a_n}{n}=\pi-\theta_a,\ldots, \lim_{n\to \infty} \frac{2\pi f_n}{n}=\pi-\theta_f,$.
Put 
 $\overline{a_n} = n-1-a_n$, $\cdots$, 
$\overline{f_n} = n-1-f_n$.  
Using these parameters, the volume of \,$T$ is given as follows.  
$$
\operatorname{Vol}(T) = 
\lim_{n\to\infty}
\,
\dfrac{\pi}{2\, n}\,
 \log 
 \left(
\left\{
\begin{matrix}
a_n & b_n & e_n \\
d_n & c_n & f_n
\end{matrix}
\right\}_{tet}\,\left\{
\begin{matrix}
\overline a_n & \overline b_n & \overline e_n \\
\overline d_n & \overline c_n & \overline f_n
\end{matrix}
\right\}_{tet}
\right).  
$$
\end{teo}  
The proof of the theorem exploits the peculiar behavior of $6j$-symbols outlined in Lemma \ref{lemma:trunc0}: they are finite sums positive real numbers each of which is growing exponentially fast. To compute the overall exponential growth of the sum is therefore sufficient to identify the summands with maximal growth rate. This key point is what makes relatively easy to compute the asymptotical behavior in our case: in general one has to deal with oscillating complex valued sums whose behavior is quite complicated. A similar property was used in \cite{C2} for $6j$-symbols associated to the representation theory of $U_q(sl_2)$ for generic $q$ to prove an analogue of the generalized volume conjecture for tetrahedra.  
The rest of the proof shows that the maximal growth rate, expressed in terms of the angles of the tetrahedron, satisfies the same Schl\"afli differential equation as the volume function, thus they differ by a constant and finally that in a special case the two functions are equal.  
\par
In Section 3, interpreting graphically the morphisms between representations in the standard way, we generalize the colored Alexander invariant to an invariant of colored graphs (which we shall denote $<\Gamma,\col>$) which is essentially equal to the invariant given in \cite{GR}.  
After defining the invariant, we construct a face model for it by using the $SL(2, \mathbb C)$ quantum $6j$-symbols along with the method already used in \cite{KR}.  
This model is a generalization of those for the Conway function and the Alexander polynomial constructed by O. Viro \cite{V} using the quantum supergroup $gl(1|1)$.
\par
The following natural evolutions of the preceding results remains open:
\begin{question}
What is the asymptotical behavior (if any) if $\theta_a,\ \ldots, \theta_f$ are complex valued? What is the geometrical meaning of such behavior? 
\end{question}
More in general similar questions may be asked for general trivalent graphs and knots. 
The fact (proved in Remark \ref{rem:kashaev}) that the Kashaev invariant of a knot can be computed as a limit with $\lambda\to\frac{n-1}{2}$ ($\lambda$ being the color of the knot) makes it natural to expect that the $SL(2, \mathbb C)$ quantum $6j$-symbol could be a good tool to investigate the volume conjecture and quantized gravity (the latter being one of the original motivations to study the $6j$-symbols). 
So it is natural to expect that the answer to the following question should be related to the generalized volume conjecture (\cite{G}):
\begin{question}
Given a hyperbolic knot $K\subset S^3$ what is (if any) the asymptotical behavior of $<K,\frac{(n-1)}{2}\lambda>_n$? 
What is its geometrical meaning?
\end{question}  
  
Other direction of study left open by the present work are related to the Turaev-Viro type invariant of $3$-manifolds introduced in \cite{GPT} using $SL(2, \mathbb{C})$ $6j$-symbols. Such invariants are defined for generic values of the parameters, 
and, specializing the parameters to half-integers, it is expected that they should be related to the logarithmic TQFT introduced in \cite{FGST} and to the logarithmic invariant in \cite{MN}. This will be a subject of future study.

\bigskip
\par\noindent
{\bf Acknowledgments.}
The authors thank Nathan Geer, Bertrand Patureau-Mirand and Roland van der Veen for helpful discussions. 
\section{Representations of ${\mathcal U}_{\xi_n}$ and their morphisms}
\setcounter{equation}{0}
In this section, we generalize the colored Alexander invariant to an invariant of oriented colored graphs.  
Such generalization was already given in \cite{GR}, and here we give another construction starting from the Clebsch-Gordan quantum coefficients (CGQC).  
\subsection{Highest weight representations of
${\mathcal U}_{q}(sl_2)$}\label{sub:1}

Let $n\in \mathbb{N}$ and let $\xi_n$ be the primitive $2n$-th root of unity
$\exp(\frac{\pi \sqrt{-1}}{n})$.  
For a complex number $a$, we will denote $\exp(\frac{\pi \sqrt{-1}\,a}{n})$ by $\xi_n^a$.  
We will also use the following notations:  
$$
\begin{aligned}
\{a\}=\xi_n^a-\xi_n^{-a}\quad (a\in \mc),&\qquad 
\{k\}! = \prod_{j=1}^k \{j\} \quad (k \in  \mathbb{N}), \ 
\\
 [a]= \frac{\{a\}}{\{1\}}, \qquad 
 &\{a,a-k\}=\prod_{j=0}^{k-1}\{a-j\}
\end{aligned}
$$
 and, if $a-b\in \{0,1,\ldots ,n-1\}$, let 
$$
\qbin{a}{b}= \prod^{a-b-1}_{j=0} \dfrac{\{a-j\} }{\{a-b-j\}}=\frac{\{a,b\}}{\{a-b\}!}.
$$
We will often use implicitly the following identities: 
\begin{itemize}
\item $\{a\}=\{n-a\},\quad (a\in \mathbb{C})$
\item $\{a\}!\{n-1-a\}!=\{n-1\}!=\sqrt{-1}^{\,n-1}\, n,\quad (a\in \{0,1,\ldots, n-1\})$ 
\item $\qbin{a}{b}=\qbin{n-1-b}{n-1-a}, \quad (a,b\in \mathbb{C}, \   (a-b)\in \{0,1,\ldots, n-1\})$ 
\item $\qbin{a}{b}=(-1)^{a-b}\qbin{a-n}{b-n}\quad (a,b\in \mathbb{C}, \  (a-b)\in \{0,1,\ldots, n-1\})$
\end{itemize}
as well as the following lemma.
\begin{lemma}
For any parameter $a$, $b$ and a non-negative integer $c$, we have
\begin{equation}
\sum_{s=0}^{c}
\xi_n^{\pm (a+b-c+2)s} \, 
{\qbin{a-s}{a-c}} \, \qbin{b+s}{b}
=
\xi_n^{\pm (b+1)c} \,\qbin{a+b+1}{a+b-c+1}.
\label{eq:c}
\end{equation}
\label{lemma:relation}
\end{lemma}
\begin{proof}{
For generic $q$, the relation \eqref{lemma:relation} is true for
the case that $a$ and $b$ are non-negative integers by (51) in \cite{K}.  
Both sides of \eqref{eq:c}
are Laurent polynomials with respect to the variable 
$q^a$, $q^b$, 
and they are
equal for any positive integers $a$ and $b$.   
Therefore, these two polynomials are
equal and the two sides of
\eqref{eq:c} coincide for any
$a$ and $b$ for generic $q$. 
Then we can specialize $q$ to $\xi_n$.   
}\end{proof}

\begin{defi}
For a parameter $q \neq 0, \pm 1$, let
${\mathcal U}_{q}(sl_2)$ be the quantized enveloping algebra of $sl_2$, which is the Hopf algebra generated by $E$, $F$, $K$ and $K^{-1}$ with relations
$$
[E,F]=\frac{K^2-K^{-2}}{{q}-q^{-1}},\  
K\, E={q}\,E\, K,\  
K\, F=q^{-1}\, F\, K,\ 
K \, K^{-1}=K^{-1}\, K=1, 
$$
and the Hopf algebra structure given by 
$$
\Delta(E)=E\otimes K+K^{-1}\otimes E,\quad 
\Delta(F)=F\otimes K+K^{-1}\otimes F,\quad 
\Delta(K^{\pm1})=K^{\pm1}\otimes K^{\pm1},
$$ 
$$
S(E)=-{q}\, E,\quad S(F)=-q^{-1}\, F,\quad S(K)=K^{-1},
$$ 
$$
\epsilon(E)=\epsilon(F)=0,\quad 
\epsilon(K)=1.
$$
\end{defi}
From now on, we will stick to the case $q=\xi_n$ unless explicitly stated the contrary.
The proof of the following is a straightforward verification of the above relations:
\begin{lemma}
For each $a\in \mc\setminus \frac{1}{2}\mz$,  there is a simple representation $V^a$ of $\,{\mathcal  U}_{\xi_n}(sl_2)$ of dimension $n$ whose basis is $\{e^a_0,\ e_1^a, \ \cdots, \ e_{n-1}^a\}$ and on which the actions of $E$, $F$ and $K$ are given by
$$
E(e^a_j)=[j]\, e^a_{j-1},\quad 
F(e^a_j)=[2a-j]\, e^a_{j+1},\quad 
K(e^a_j)=\xi_n^{a-j}\, e^a_j \quad
(e_{-1}^a = e_{n}^a = 0).  
$$
The basis $\{e_0^a$, $e_1^a$, $\cdots$, $e_{n-1}^a\}$ of $V^a$ will be called the {\it weight basis} of $V^a$.  
Two such representations $V^a$ and $V^b$ are isomorphic iff $a-b\in 2n\mz$.
The representation $(V^a)^*$ is isomorphic to $V^{n-1-a}$,
a duality pairing realizing this isomorphism being: 
\begin{equation}
\cap_{a,b}(e^a_i,\, e^b_j)
=
 \delta_{b,n-1-a}\,\delta_{i,n-1-j} \, \xi_n^{-(a-i)(n-1)}.
\label{equation:cap}
\end{equation}
Similarly, an invariant vector in $V^a\otimes V^{b}$ is given by: 
\begin{equation}
\cup_{a,b}
=
\delta_{b,n-1-a}
\sum_{i=0}^{n-1}\xi_n^{(b-n+1+i)(n-1)}\, e^a_i\otimes e^{b}_{n-1-i}.  
\label{equation:cup}
\end{equation}
\end{lemma}

In what follows we will represent graphically the maps $\cap_{a,b}$ and $\cup_{a,b}$ as follows:
\begin{equation}
\begin{matrix}
\begin{picture}(40,30)
\thicklines
\put(20, 30){\line(-1,-1){20}}
\put(20, 30){\line(1,-1){20}}
\put(0,10){\vector(0,-1){10}}
\put(40,10){\vector(0,-1){10}}
\put(0, 20){$a$}
\put(35, 20){$n-1-a$}
\end{picture}
\\ 
\qquad i \qquad n-1-i
\end{matrix}\qquad,
\qquad
\begin{matrix}
\qquad i \qquad n-1-i
\\ 
\begin{picture}(40,30)
\thicklines
\put(20, 0){\line(-1,1){20}}
\put(20, 0){\line(1,1){20}}
\put(0,30){\vector(0,-1){10}}
\put(40,30){\vector(0,-1){10}}
\put(0, 5){$a$}
\put(35, 5){$n-1-a$}
\end{picture}
\end{matrix}\qquad.
\end{equation}
\subsection{The $R$-matrix}
The $R$-matrix corresponding to the colored Alexander invariant is given in \cite{ADO}, and is also used in \cite{M}.  
The construction of the representation of ${\mathcal U}_{\xi_n}(sl_2)$ is a little bit different from that in \cite{M}, and 
we define ${}_a^bR : V^a \otimes V^b \to V^b \otimes V^a$ (the morphism obtained by composing the $R$-matrix action with the flip) as follows:   
\begin{multline}
{}_a^bR(e^{a}_{u}\otimes e^{b}_{v})
=
\\
\sum_{m} \{m\}!
\xi_n^{2 (a - u)(b - v) - m (a - b - u + v) -\frac{m(m+1)}{2}}
\qbin{u}{\!\!u-m\!\!\!} \qbin{2b-v}{\!\!2b-v-m\!\!\!} 
e^{b}_{v+m}\otimes e^{a}_{u-m}.
\label{equation:R}
\end{multline}
where $m$ ranges in $[0,{ \min}(n-v-1,u)]\cap \mathbb{N}$.
We denote ${}^{b}_{a}R_{u,v}^{h,k}$ the coefficient of $R(e^{a}_u\otimes e^{b}_v)$ with respect to $e^{b}_h\otimes e^{a}_k$.
\begin{prop}\label{prop:rmatrix}
The morphism 
${}_a^bR$ 
given above is the $R$-matrix of the non-integral representations, in other words, ${}_a^bR$ satisfies 
\begin{equation}
{}_a^bR \, \Delta(x) = \Delta(x) \, {}_a^bR
\label{eq:Rdelta}
\end{equation}
as mappings from $V^a \otimes V^b$ to $V^b \otimes V^a$ 
for any $x \in {\mathcal U}_{\xi_n}(sl_2)$,
and
\begin{equation}
({}_b^cR \otimes \id)\, (\id \otimes {}_a^cR) \, 
({}_a^bR \otimes \id)
=
 (\id \otimes {}_a^bR) \, 
({}_a^cR \otimes \id)\, ( \id\otimes{}_b^cR )
\label{eq:YBE}
\end{equation}
as mappings from 
$V^a \otimes V^b \otimes V^c$ to
$V^c \otimes V^b \otimes V^a$.  
\end{prop}
\begin{prof}
It is easy to check \eqref{eq:Rdelta} for all the generators $K$, $E$ and $F$.  
The relation \eqref{eq:YBE} comes from the braid relation (3.8) of \cite{ADO} since our ${}_a^bR$ is related to 
$G(\alpha, \beta, +)$ defined by (3.5) in \cite{ADO} as follows. 
Using the superscript $ADO$ for symbols introduced in (3.1)-(3.3) of \cite{ADO} one sees that $(\xi^2,m)^{ADO}_{\xi^2}=(-1)^m\xi^{\frac{m(m+1)}{2}}\{m\}!$ and:
$$
\left[\begin{array}{c}
m \\
n \end{array}\right]^{ADO}_{\xi^2}=\qbin{m}{n}\xi^{n(m-n)},\  \left(\begin{array}{c}
m \\
n \end{array}\right)^{ADO}_{\xi^{-4b},\xi^2}=\{2b-n,2b-m\}\xi^{\frac{(m-n)(m+n-1-4b)}{2}}
$$ 
Observing formula (3.5) of \cite{ADO},  one easily sees that if the matrix $G(\alpha,\beta,+)$ satisfies the braid relation, then it will keep satisfying it for any choice of $\mu,\nu,f,F,\eta,\kappa$. Thus we choose $\omega=\xi^2,\alpha= \xi_n^{-4a}$, $\beta=\xi_n^{-4b}$, $\mu=\nu = 1/4$,  $f=F=1$, and $\eta=\kappa=0$ so that $G(\xi_n^{-4a}, \xi_n^{-4b}, +)$ is given by
\begin{multline}
G_{v+m,u-m}^{u,v}(\xi_n^{-4a}, \xi_n^{-4b}, +)
=
\\
\xi_n^{2vu-2bu-2a v +m(-a+u+b-v)-m(m+1)/2} \,
\qbin{u}{u-m} \, 
\{2b-v, 2b-v-m\} 
\end{multline}
and so 
$$
{}_a^bR_{uv}^{v+m\,u-m}
=
\xi_n^{2ab} \, G_{v+m,u-m}^{u,v}(\xi_n^{-4a}, \xi_n^{-4b}, +).
$$
Hence ${}_a^bR$ satisfies the braid relation \eqref{eq:YBE}.  \qed
\end{prof}
The $R$-matrix given by \eqref{equation:R} is represented graphically as follows.  
$$
{}_a^bR_{kl}^{ij} : \quad
\begin{matrix}
i \qquad j \\
\setlength{\unitlength}{0.2mm}
\begin{picture}(50, 50)
\thicklines
\put(50,50){\vector(-1,-1){50}}
\put(0,50){\line(1,-1){20}}
\put(30,20){\vector(1,-1){20}}
\put(-8, 12){$a$}
\put(48, 12){$b$}
\end{picture}
\\
k \qquad l
\end{matrix}\ ,
\qquad
\left({}_a^b R^{-1}\right)_{kl}^{ij} : \quad
\begin{matrix}
i \qquad j \\
\setlength{\unitlength}{0.2mm}
\begin{picture}(50, 50)
\thicklines
\put(0,50){\vector(1,-1){50}}
\put(50,50){\line(-1,-1){20}}
\put(20,20){\vector(-1,-1){20}}
\put(-8, 12){$a$}
\put(48, 12){$b$}
\end{picture}
\\
k \qquad l
\end{matrix}\ .  
$$

\subsection{Clebsch-Gordan quantum coefficients}
Let us consider the tensor product
$V^a \otimes V^b$.  
By using standard arguments on the weight
space decomposition, one can prove the following
decomposition of the tensor product: 
\begin{prop}
Let $V^a$, $V^b$ be highest weight
representations of non-half-integer parameters $a,b$.
If $a+b$ is not a half-integer, then
\begin{equation}
V^a \otimes V^b
=
\bigoplus_{a+b-c=0,
1,
\cdots, n-1} V^c.
\end{equation}
\label{prop:decomposition}
\end{prop}
The weight basis 
$e_t^c$
of  
$V_c$
is a linear combination of the tensors $e_u^a \otimes e_v^b$
of
the weight basis of $V^a$ and $V^b$.  
An explicit computation of the coefficients expressing $e^{c}_t$ in terms of $e^a_u\otimes e^b_v$ is provided by the following (see Appendix \ref{app:Clebsch-Gordan} for a proof):  
\begin{teo}[Clebsch-Gordan decomposition]
\label{teo:ClebschGordan}
If $a+b-c\in \{0$, $1$, $\cdots$, $n-1\}$, 
any ${\mathcal U}_{\xi_n}(sl_2)$ module map
$\iota_c^{a, b} : V^c \longrightarrow V^a\otimes V^b$ is a scalar multiple of the inclusion map 
$Y_{c}^{a,b}:V^c\to V^{a}\otimes V^b$ given by 
$$
Y_c^{a,b}(e^c_t)
=
\sum_{u+v-t=a+b-c} C^{a,b,c}_{u,v,t}\, e^a_u\otimes e^b_v, 
$$ 
where $C_{u,v,t}^{a,b,c}$ is given by:
   \begin{multline}
\sqrt{-1}^{c-a-b} \,
(-1)^{(v-t)} \,
\xi_n^\frac{v(2b-v+1)-u(2a-u+1)}{2} \,
\qbin{2c}{2c-t}^{-1}\qbin{2c}{a+b+c-(n-1)}
\\
\sum_{z+w=t}\, 
(-1)^z \, 
\xi_n^\frac{(2z-t)(2c-t+1)}{2} \,
\left[\begin{matrix}
a+b-c \\ u-z
\end{matrix}
\right] \,
\left[\begin{matrix}
2a-u+z \\ 2a-u
\end{matrix} \right]\,
\left[\begin{matrix}
2b-v+w \\ 2b-v
\end{matrix} 
\right]. 
\label{equation:QCGC}
\end{multline}
\end{teo}
The coefficient
$C_{u,v,t}^{a,b,c}$ is called the 
{\it Clebsch-Gordan quantum coefficient} (CGQC).  
\begin{figure}[htb]
$$
\begin{matrix}
c & & \quad c & & c \quad
\\
\raisebox{-5mm}{\includegraphics[scale=0.5]{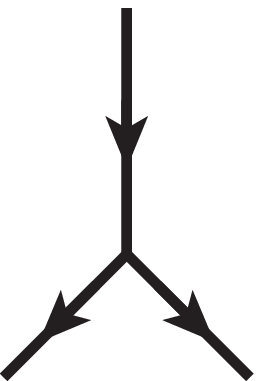}}
& = &
\raisebox{-5mm}{\includegraphics[scale=0.5]{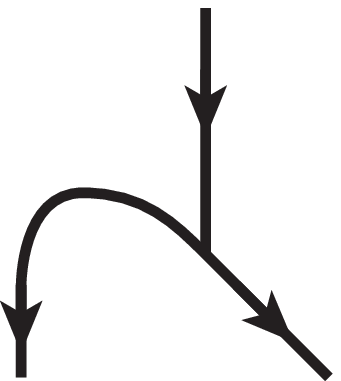}}
& = &
\raisebox{-5mm}{\includegraphics[scale=0.5]{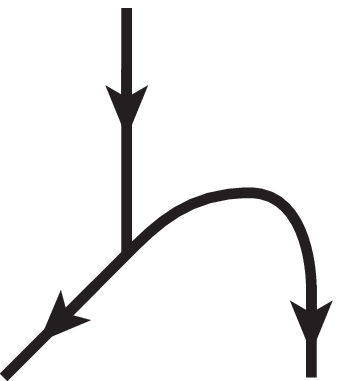}}
\\
a \qquad\quad b & & a \qquad\quad b & & a \qquad\quad b 
\\
Y_{a,b}^c &  & L_{a,b}^c  & & R_{a,b}^c
\end{matrix}
$$
\caption{The first equality is the definition of $Y^{c}_{a,b}$, the second is Lemma \ref{lem:nonordered}.}\label{fig:moveleg}
\end{figure}
In order to get invariants of unoriented graphs,  the operators associated to the other elementary graphs must constructed by ``moving a leg of the $Y$-shaped graph up or down" (see Figure \ref{fig:moveleg}). 
Hence we define projectors $L_{a,b}^c:V^a\otimes V^b\to V^c$ out of $Y^{a,b}_c$ as:  $L_{a,b}^c= (\cap_{a,n-1-a}\otimes Id_c)\circ (Id_a\otimes Y_{b}^{n-1-a,c})$. So letting $L_{a,b}^c(e^a_{u}\otimes e^b_{v})=\sum_tL^{a,b,c}_{u,v,t} e^c_t$, the coefficients are explicitly given by 
\begin{equation}\label{eq:p}
L^{a,b,c}_{u,v,t}
= 
C^{n-1-a,c,b}_{n-1-u,t,v}\, \xi_n^{-a(n-1)}\, \xi_n^{(n-1)u}
\end{equation}
One may also define $R_{a,b}^{c}$ by ``pulling-up the right leg" i.e. by setting 
$
R_{a,b}^c
=
 (Id_c\otimes \cap_{n-1-b,b})\circ (Y_{a}^{c,n-1-b}\otimes Id_b)
 $ 
 (as in the r.h.s. of Figure \ref{fig:moveleg}), 
 but as the following lemma shows, the two choices are equivalent (see Appendix \ref{app:bend} for a proof).
\begin{lemma}\label{lem:nonordered}
It holds that
$$
C^{n-1-a,c,b}_{n-1-u,t,v}\,  \xi_n^{-a(n-1)}\, \xi_n^{(n-1)u}
=
C^{c,n-1-b,a}_{t,n-1-v,u}\, \xi_n^{-(n-1-b)(n-1)}\, \xi_n^{(n-1)(n-1-v)}. 
$$
\label{lemma:bend}
\end{lemma}
Therefore we have well defined projectors $Y_{a,b}^c:V^a\otimes V^b\to V^c$ (about the notation: $Y^{a,b}_c$ and $Y_{a,b}^c$ are distinguished by the position of the indices). In order to have an explicit formula, applying two times Lemma \ref{lemma:bend} we get the following:   
\begin{prop}\label{prop:Yprojectors}
The projection $Y_{a, b}^c : V^a \otimes V^b \to V^c$ is given by
\begin{equation}
Y_{a, b}^c(e_u^a \otimes e_v^b)
=
\sum_t C_{n-1-v, n-1-u, n-1-t}^{n-1-b, n-1-a, n-1-c}\  e_t^c.  
\label{equation:projection}
\end{equation}
\end{prop}
In what follows, the Clebsch-Gordan quantum coefficients given by \eqref{equation:QCGC} will be represented graphically as follows:
$$
C_{m_1,m_2,m}^{a,b,c}
:
\begin{matrix}
m_1 \qquad\qquad m_2
\\
\setlength{\unitlength}{0.3mm}
\begin{picture}(50,50)
\thicklines
\put(0,50){\vector(1,-1){25}}
\put(50,50){\vector(-1,-1){25}}
\put(25,25){\vector(0,-1){25}}
\put(1,30){$a$}
\put(40,30){$b$}
\put(14,7){$c$}
\end{picture}
\\ m
\end{matrix},
\quad
C_{n-1-m_2,n-1-m_1,n-1-m}^{n-1-b,n-1-a,n-1-c}
:
\begin{matrix}
m
\\
\setlength{\unitlength}{0.3mm}
\begin{picture}(50,50)
\thicklines
\put(25,25){\vector(-1,-1){25}}
\put(25,25){\vector(1,-1){25}}
\put(25,50){\vector(0,-1){25}}
\put(1,15){$a$}
\put(40,15){$b$}
\put(14,35){$c$}
\end{picture}
\\
m_1 \qquad\qquad m_2
\end{matrix}.
$$ 
(Here in the right part we use Proposition \ref{prop:Yprojectors} to rewrite $Y_{a,b}^c$.)  
The following Lemma is proved in Appendix \ref{app:thetaval}.
\begin{lemma}\label{lem:thetagraph}
It holds:
\begin{equation}\label{eq:theta}
\raisebox{-0.8cm}
{\psfrag{a}{$a$}\psfrag{b}{$b$}\psfrag{c}{$c$}
 \includegraphics[width=0.8cm]{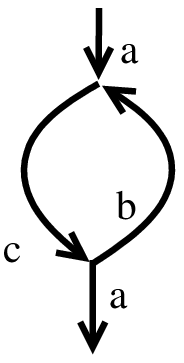}}=\qbin{2a+n}{2a+1}Id_a 
\end{equation} 

\end{lemma}
The decomposition of $V^a \otimes V^b$ is expressed by  $Y_c^{a, b}$ and $Y_{a, b}^{c}$ as follows.  
\begin{prop}
Let
$\operatorname{id}$ be the identity operator on $V^a \otimes V^b$
Then 
\begin{equation}
\operatorname{id}= 
\sum_{c\,:\,  a+b-c = 0, 1, \cdots, n-1}
\qbin{2c+n}{2c+1}^{-1}
Y_c^{a,b} \, Y_{a, b}^c.  
\label{eq:decomposition}
\end{equation}
\end{prop}
\begin{proof}{
Formula \eqref{eq:theta} implies that
$
Y_c^{a,b}\,  Y_{a, b}^c \, Y_c^{a,b} \, Y_{a, b}^c
=
\qbin{2c+n}{2c+1} \, Y_c^{a,b\,}  Y_{a, b}^c.  
$
 Therefore, $\qbin{2c+n}{2c+1}^{-1}\, 
Y_c^{a,b} \, Y_{a, b}^c$ is the identity on the
subspace of $V^a \otimes V^b$ isomorphic to $V^c$.  
}
\end{proof}
\subsection{Quantum $6j$-symbols}
In this subsection, we compute the $SL(2, \mathbb C)$ quantum $6j$-symbol actually by using the CGQC introduced previously.  
The quantum $6j$-symbol is defined by the relation in Figure \ref{fig:6j}.
The left diagram represents the composition of two inclusions
$V^j \to V^{j_{12}} \otimes V^{j_3}$ and $V^{j_{12}} \to V^{j_1} \otimes V^{j_2}$ while the right diagram represents the composition of two inclusions
$V^j \to V^{j_{1}} \otimes V^{j_{23}}$ and $V^{j_{23}} \to V^{j_2} \otimes V^{j_3}$.  
Let $\iota_l,\iota_r$ be the resulting maps.
%
\begin{figure}[b]
$$
\begin{matrix}
j_1 \qquad j_2  \qquad j_3 \\
\includegraphics[scale=0.3]{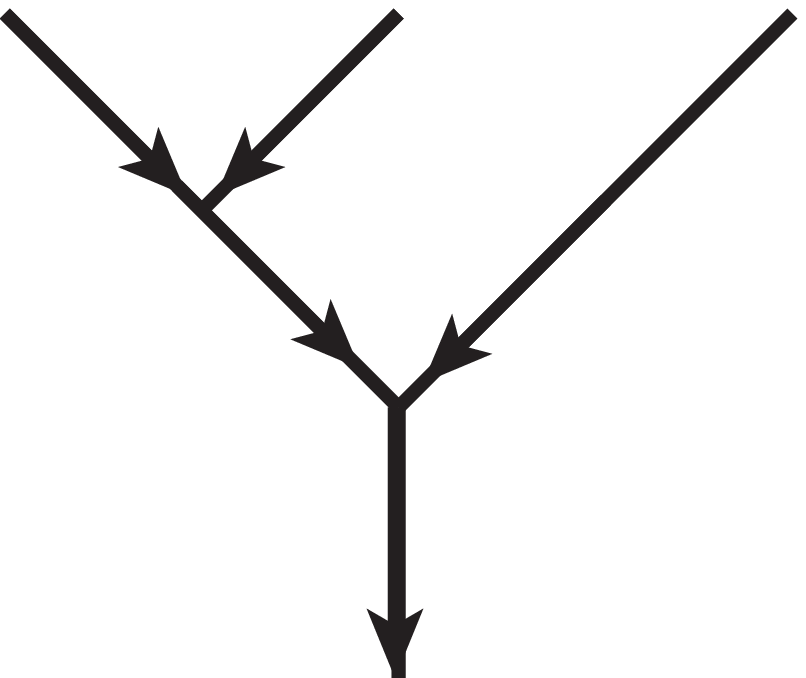} \\
j
\end{matrix}
\hspace{-2cm} j_{12}
\hspace{17mm}
=
\sum_{j_{23}}
\left\{
\begin{matrix}
j_1 & j_2 & j_{12} \\
j_3 & j & j_{23}
\end{matrix}
\right\}_{\xi_n} \,
\begin{matrix}
j_1 \qquad j_2  \qquad j_3 \\
\includegraphics[scale=0.3]{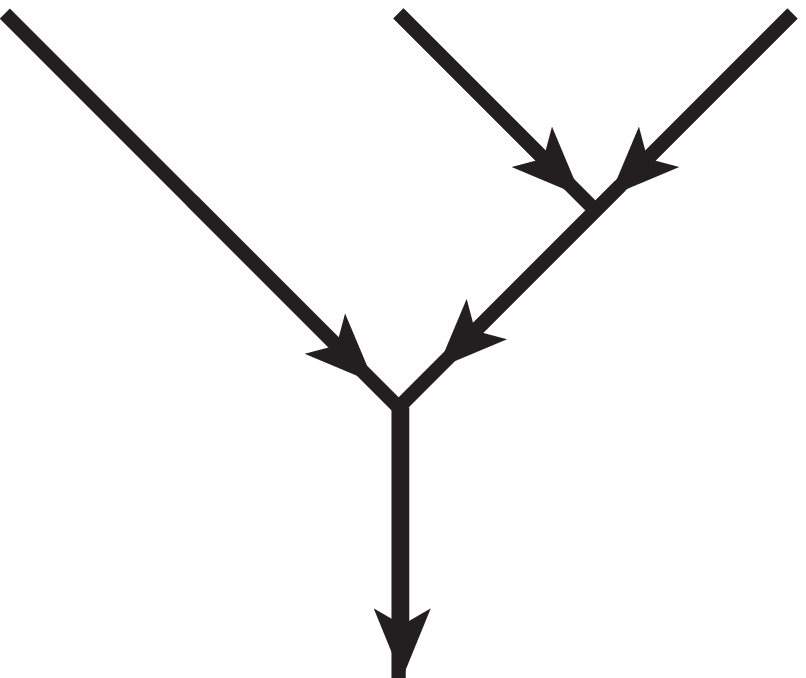} \\
j
\end{matrix}
\hspace{-1cm} j_{23}
\hspace{7mm}
$$
\caption{The quantum $6j$-symbol is defined by the above diagram.}
\label{fig:6j}
\end{figure}
Then Figure \ref{fig:6j} translates the equality
$
\iota_l(v) = 
\sum_{j_{23}}
\left\{
\begin{matrix}
j_1 & j_2 & j_{12} \\
j_3 & j & j_{23}
\end{matrix}
\right\}_{\xi_n} 
\iota_r(v)
$
for $v \in V^j$. 
The quantum $6j$-symbols for non-integral highest weight representations are given as follows (see Appendix \ref{app:theo6j} for a proof). 

\begin{teo}\label{theo:6j}
For $a$, $b$, $\cdots$, $f\in \mathbb{C}\setminus \frac{1}{2}\mathbb{Z}$ satisfying
$
a+b-e, a+f-c, b+d-f, d+e-c \in \bZ
$,
\begin{multline}
\left\{
\begin{matrix}
a & b & e \\
d & c & f
\end{matrix}
\right\}_{\xi_n}
=
\\
(-1)^{n-1+B_{afc}} \, 
{\qbin{2f+n}{2f+1}}^{-1} \, 
\dfrac{\{B_{dec}\}! \, \{B_{abe}\}!}
{\{B_{bdf}\}! \, \{B_{afc}\}!} \, 
\qbin{2e}{A_{abe}+1-n} \, 
{\qbin{2e}{B_{ecd}}}^{-1} \,
\\
\sum_{z = \max(0, -B_{bdf}+B_{dec})}^{\min(B_{dec}, B_{afc})}\!\!\!\!\!\!\!\!\!\!\!\!\!\!\!\!
(-1)^z 
\qbin{A_{afc}+1}{\!\!2c+z+1\!\!\!} \!
\qbin{\!\!B_{acf}+z\!\!\!}{B_{acf}} \!
\qbin{\!\!B_{bfd}+B_{dec}-z\!\!\!}{B_{bfd}} \!
\qbin{\!\!B_{dce}+z\!\!\!}{B_{dfb}}\!, 
\label{equation:6j}
\end{multline}
where
\begin{equation}
A_{xyz} = x + y + z,
\qquad
B_{xyz} = x+y-z.  
\label{eq:ABC}
\end{equation}
\end{teo}
\par\noindent
{\bf Remark.}
These $6j$-symbols were already computed in \cite{GP}.
%
%
%

\subsection{Values of tetrahedra}
The following is a straightforward consequence of the definition of $6j$-symbols and of Lemma \ref{lem:thetagraph}:
\begin{lemma} It holds:
\begin{equation}\raisebox{-1.2cm}{
\psfrag{a}{$a$}\psfrag{b}{$b$}\psfrag{c}{$c$}\psfrag{d}{$d$}\psfrag{e}{$e$}\psfrag{f}{$f$}
\includegraphics[width=1.2cm]{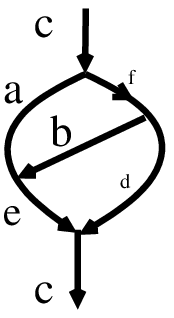}}=\left\{
\begin{matrix}
a&b&e\\
d&c&f
\end{matrix}
\right\}_{\xi_n}\qbin{2f+n}{2f+1} \raisebox{-1.2cm}{
\psfrag{a}{$c$}\psfrag{b}{$f$}\psfrag{c}{$a$}
\includegraphics[width=1.2cm]{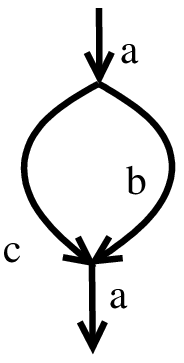}
}
\end{equation}
\end{lemma}
Let us define 
\begin{equation}
\left\{
\begin{matrix}
a&b&e\\
d&c&f
\end{matrix}
\right\}_{tet}
=
\left\{
\begin{matrix}
a&b&e\\
d&c&f
\end{matrix}
\right\}_{\xi_n}\qbin{2f+n}{2f+1}.
\label{equation:6jtet}
\end{equation}
We will prove later in much greater generality (Theorem \ref{teo:invariance}) that $\left\{
\begin{matrix}
a&b&e\\
d&c&f
\end{matrix}
\right\}_{tet}$ is the value of an invariant for the tetrahedron of Figure \ref{fig:tet} and thus it has all the symmetries of the tetrahedron (up to switching the color of an edge with its complement to $n-1$ if the symmetry changes the orientation of the edge).

\begin{figure}[htb]
$$
\includegraphics[scale=0.7]{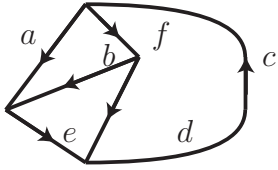}
\hspace{-3.1cm}\raisebox{1.6cm}{$a\qquad\ \ \ \ \  f$}
\hspace{-0.9cm}\raisebox{1.3cm}{$b\quad\qquad\quad\ \   c$}
\hspace{-2.8cm}\raisebox{0.3cm}{$e\quad\quad\quad \, d$}
$$
\caption{Colored oriented tetrahedral graph}
\label{fig:tet}
\end{figure} 
Using formula \eqref{eq:ABC}, if in $\left\{
\begin{matrix}
a&b&e\\
d&c&f
\end{matrix}
\right\}_{tet}$ one fixes the values of $B_{dec},B_{abc},B_{bdf}$ and $B_{afc}$ in $\{0,1,\ldots, n-1\}$ and computes the values of $b,f,c$ as functions of these values and of $e,a,d$, then one gets a rational function of $\xi_n^a,\xi_n^d$ and $\xi_n^e$ with complex coefficients, which is actually a Laurent polynomial of $\xi_n^a$ and $\xi_n^d$. Similarly one can express all the variables in terms of $e,b,c$ or $e,f,c$ showing that the function is a Laurent polynomial also of $\xi_n^b,\xi_n^c,\xi_n^f$. To prove that it is also a Laurent polynomial of $\xi_n^e$,  we exploit the following symmetry (induced by an isotopy rotating the picture of Figure \ref{fig:tet} of $180^{\circ}$ along an axis contained in the blackboard)
$$
\left\{
\begin{matrix}
a & b & e \\
d & c & f
\end{matrix}
\right\}_{tet}
=
\left\{
\begin{matrix}
f & n-1-b & d \\
e & c & a
\end{matrix}
\right\}_{tet}, 
$$ 
and apply the same argument to the right-hand side to conclude that the l.h.s. is a Laurent polynomial with respect to $\xi_n^e$.
Hence the value of the tetrahedron is holomorphic with respect to the parameters $a$, $\cdots$, $f$. This implies that 
$
\left\{
\begin{matrix}
f & b & d \\
e & c & a
\end{matrix}
\right\}_{tet} 
$ 
is well-defined even if some of parameters are half-integers, therefore from now on we will allow half-integers values for the parameters.
So, in order to exploit this, let us set a definition.
\begin{defi}
A triple of three integers $(i,j,k)$ is called {\it admissible} if they satisfy the following conditions:
\begin{multline}
0 < i,\ j,\ k < n-1, \quad
n-1 < i+j+k < 2\, (n-1), 
\\
0 < i+j-k,\ j+k-i,\ k+i-j < n-1.
\end{multline}
(here $n$ is such that $\xi=exp(\frac{i\pi}{n})$). 
\end{defi}
The following lemma will be crucial in the proof of Theorem \ref{th:trunc}:
\begin{lemma}\label{lemma:trunc0}
Let $a$, $b$, $c$, $d$, $e$, $f$ be integers such that $(a, b, e)$, $(a, c, f)$, $(b, d, f)$, $(c, d, e)$ are admissible triples.  
Then it holds:
\begin{multline}
\left\{
\begin{matrix}
a & b & e \\
d & c & f 
\end{matrix}\right\}_{tet}
=
(-1)^{n-1} \, 
\dfrac{\{B_{dec}\}! \, \{B_{abe}\}!}
{\{B_{bdf}\}! \, \{B_{afc}\}!} \, 
\qbin{2e}{A_{abe}+1-n} \, 
{\qbin{2e}{B_{ecd}}}^{-1} \,
\\
\sum_{z = m}^{M}\,
\qbin{A_{afc}+1-n}{2c+z+1-n} \, 
\qbin{B_{acf}+z}{B_{acf}} \,
\qbin{B_{bfd}+B_{dec}-z}{B_{bfd}} \,
\qbin{B_{dce}+z}{B_{dfb}},        
\label{equation:trunc0}
\end{multline}
where 
$$
\begin{aligned}
m &= \max(0,\ n-1-2c,\  b-c+e+f-n+1,\  -b-c+e+f), \\
M &= \min(B_{dec},\ B_{afc},\  n-1-B_{acf},\  n-1-B_{dce}).
\end{aligned}
$$  
Let
$$
R(z)
=
\qbin{A_{afc}+1-n}{2c+z+1-n} \, 
\qbin{B_{acf}+z}{B_{acf}} \,
\qbin{B_{bfd}+B_{dec}-z}{B_{bfd}} \,
\qbin{B_{dce}+z}{B_{dfb}}        
$$
for $m < z < M$.  
Then $R(z)$ is positive, and there is a unique integer
$z_0$ satisfies $r(z_0) \geq 1$ and $r(z_0+1) \leq 1$ 
for
\begin{equation*}
r(z) = \dfrac{R(z)}{R(z-1)}
=
\dfrac{\{B_{afc}-z+1\} \, \{B_{acf}+z\}\,
\{B_{dec}-z+1\} \, \{B_{dce}+z\}}
{\{2c+z+1-n\}  \{z\} 
 \{ B_{bfd}+B_{dec}-z+1\} 
\{B_{dce} - B_{dfb} + z\} }.  
\end{equation*}
This $z_0$ satisfies
\begin{equation}
R(z_0)
=
\max\{R(z) \mid z\in \bZ,\  m < z < M\}
.
\label{equation:Rmax}
\end{equation}
\end{lemma}
\begin{proof}
The formula \eqref{equation:trunc0} comes from \eqref{equation:6j}, \eqref{equation:6jtet} and the following relation
$$(-1)^z\,\qbin{A_{afc}+1}{2c+z+1}=
(-1)^{B_{afc}} \, \qbin{A_{afc}+1-n}{2c+z+1-n}.
$$ 
Observe that since the colors are integers the summation range $[m,M]\subset [0,n-1]$ is the set of values of $z$ in formula \ref{equation:6j} such that all the binomials are non-zero. Moreover, for $z$ satisfying $m \leq z \leq M$, the four binomials in $R(z)$ are all positive.  
\par
For $z\in ]m,M]$ we rewrite $r(z)$ as follows:
\begin{multline}
r(z)=\frac{\sin(\frac{\pi}{n}(B_{afc}-z+1))\sin(\frac{\pi}{n}(B_{acf}+z))}{\sin(\frac{\pi}{n}(2c+z+1-n))\sin(\frac{\pi}{n}z)} \times \\
\times \frac{\sin(\frac{\pi}{n}(B_{dec}-z+1))\sin(\frac{\pi}{n}(B_{dce}+z))} {\sin(\frac{\pi}{n}(z-(B_{bfd}+B_{dec}-n+1)))\sin(\frac{\pi}{n}(B_{dce} - B_{dfb} + z))}
\end{multline}

Observe also that, since all colors are integers, $r(z)$ is continuous and real-valued for $z\in ]m,M]$ and since $r(z)>1$ if $z$ is close to $m$ and $0 < r(z)<1$ if $z$ is close to $M$, there exist a solution $x_1\in [m,M]$ of $r(z)=1$; thus $R(z)$ attains a local maximum at $z_0=\lfloor x_1\rfloor$.  
We will now show that $x_1$ is actually the only maximum of $R$ on the interval $[m,M]$ and this will conclude the proof. To prove this, we now show that there is only one solution to $r(z)=1$ in $]m,M]$ by proving that $r'(z)<0 \ \forall z\in ]m,M]$. It holds $r'(z)=\frac{\pi}{n} r(z)f(z)$ where:
\begin{multline}
f(z)=-\ctg(\frac{\pi}{n}(B_{afc}-z+1))-\ctg(\frac{\pi}{n}(z-(n-1-2c)))-\ctg(\frac{\pi}{n}(B_{dec}-z+1))+\\
-\ctg(\frac{\pi}{n}( z-(B_{bfd}+B_{dec}-n+1)))+\ctg(\frac{\pi}{n}(B_{acf}+z))-\ctg(\frac{\pi}{n}z)+\\
+\ctg(\frac{\pi}{n}(B_{dce}+z))-\ctg(\frac{\pi}{n}(B_{dce} - B_{dfb} + z))
\end{multline}
To conclude it is then sufficient to observe that for $z\in ]m,M]$ the following inequalities hold: 
\begin{align}
-\ctg(\frac{\pi}{n}(B_{afc}-z+1))-\ctg(\frac{\pi}{n}(z-(n-1-2c)))<0\\
-\ctg(\frac{\pi}{n}(B_{dec}-z+1))-\ctg(\frac{\pi}{n}( z-(B_{bfd}+B_{dec}-n+1)))<0\\
\ctg(\frac{\pi}{n}(B_{acf}+z))-\ctg(\frac{\pi}{n}z)<0\\
\ctg(\frac{\pi}{n}(B_{dce}+z))-\ctg(\frac{\pi}{n}(B_{dce} - B_{dfb} + z))<0
\end{align}
Indeed the latter two hold because for all $0<y<x<\pi$ it holds $ctg(x)-ctg(y)<0$, and the former two because for all $0<x,y<\pi$ such that $x+y<\pi$ it holds $-ctg(x)-ctg(y)<0$. Thus $f(z)<0$ and so $r'(z)<0,\ \forall z\in ]m,M[$.
This proves that $z_0=\lfloor x_1\rfloor$ is the only maximum of $R(z)$ on $[m,M]$.
\end{proof}
\par
\bigskip

\subsection{Relations satisfied by the tetrahedra}
The proof of the relations is similar to that of the analogous relations in the generic $q$ case; the symmetries are a consequence of Theorem \ref{teo:invariance}.
\par
{\bf Orthogonality relation:}
$$\sum_f
\qbin{2f + n}{2f+1}^{-1} \, 
\qbin{2g + n}{2g+1}^{-1} \, 
\left\{
\begin{matrix}
a & b & e \\
d & c & f
\end{matrix}
\right\}_{tet} \,
\left\{
\begin{matrix}
d & b & f \\
a & c & g
\end{matrix}
\right\}_{tet}
=
\delta_{eg}, 
$$
where $f$ ranges over all the complex numbers such that both $b+d-f$ and $f+a-c$ are in $\{0,1,\ldots, n-1\}$. 
\par
{\bf Pentagon relation:}
\begin{multline}
\sum_h
\qbin{2h+n}{2h+1}^{-1} \, 
\left\{
\begin{matrix}
a & b & f \\
g & c &h
\end{matrix}
\right\}_{tet} \,
\left\{
\begin{matrix}
a & h & g \\
e & d & i
\end{matrix}
\right\}_{tet} \,
\left\{
\begin{matrix}
b & c & h \\
d & i & j
\end{matrix}
\right\}_{tet} 
\\
=
\left\{
\begin{matrix}
f & c & g \\
d & e & j
\end{matrix}
\right\}_{tet} \,
\left\{
\begin{matrix}
a & b & f \\
j & e & i
\end{matrix}
\right\}_{tet}.  
\end{multline}
where $h$ ranges over all the complex numbers such that all of $h+a-c$, $g+b-h
$, $e+h-i$, $a+h-g$, $b+c-h$, and $h+d-i$ are in $\{0,1,\ldots, n-1\}$. 
\par
{\bf Symmetry:}
Since the change of the orientation of an edge colored by $i$ corresponds to the change of the color $i$ to $\overline i = n-1-i$, 
we have a group of $24$ symmetries (induced by the invariance under isotopy) and (for instance) generated by the first $3$ here below (we also specify some other elements of the group which will be used later):  
\begin{multline}
\left\{
\begin{matrix}
a & b & e \\
d & c & f
\end{matrix}
\right\}_{tet} 
=
\left\{
\begin{matrix}
b & \overline e & \overline a \\
c & f & d
\end{matrix}
\right\}_{tet} 
=
\left\{
\begin{matrix}
f & \overline b & d \\
e & c & a
\end{matrix}
\right\}_{tet} 
=
\left\{
\begin{matrix}
\overline d & \overline b & \overline f \\
\overline a & \overline c & \overline e
\end{matrix}
\right\}_{tet} 
=
\\
\left\{
\begin{matrix}
c & \overline f & a \\
b & e & \overline d
\end{matrix}
\right\}_{tet} 
=
\left\{
\begin{matrix}
d & e & c \\
\overline a & f & b
\end{matrix}
\right\}_{tet} 
=
\left\{
\begin{matrix}
e & d & c \\
\overline f & a & \overline b
\end{matrix}
\right\}_{tet} 
.  
\label{equation:symmetry}
\end{multline}
\section{Relation between the  $6j$-symbol and the hyperbolic volume}
\setcounter{equation}{0}
In this section, we investigate the relation between 
$
\left\{
\begin{matrix}
a & b & e \\
d & c & f
\end{matrix}
\right\}_{tet}
$ 
and the hyperbolic volume of an ideal or truncated hyperbolic tetrahedron.  
\subsection{Volume of an ideal tetrahedron}
First, we consider an ideal tetrahedron $T$ with dihedral angles $\alpha$, $\beta$ and $\gamma$ satisfying $a+\beta+\gamma=\pi$.  
Let $a_n,b_n,c_n$ be sequences of integers such that $\lim_{n\to \infty} \frac{2\pi a_n}{n}=\pi-\alpha$, $\lim_{n\to \infty} \frac{2\pi b_n}{n}=\pi-\beta$, $\lim_{n\to \infty} \frac{2\pi c_n}{n}=\pi-\gamma$ and $a_n+b_n+c_n=n-1$.
Then we have
\begin{multline*}
\left\{
\begin{matrix}
a_n & b_n & c_n \\
a_n & b_n & c_n
\end{matrix}
\right\}_{tet}
=
{\qbin{2c_n}{n-1-2a_n}}^{-1}
\sum_{z = \max(0, 2(c_n-b_n))}^{n-1-2b_n}
\\
(-1)^z 
\qbin{n}{\!\!
2b_n+z+1 \!\!
} 
\qbin{\!\!
n-1-2c_n+z \!\!
}{n-1-2c_n}
\qbin{2c_n-z}{\!\!
n-1-2a_n \!\!
} 
\qbin{\!\!
n-1-2c_n+z \!\!
}{n-1-2b_n}.    
\end{multline*}
The summand does not vanish only when $z = n-1-2b_n$ and, using $\{n-1\}! = \sqrt{-1}^{\, n-1}\,n$, we get
\begin{equation*}
\begin{aligned}
&\left\{
\begin{matrix}
a_n & b_n & c_n \\
a_n & b_n &c_n
\end{matrix}
\right\}_{tet}
\!\!
=
(-1)^{n-1} \, 
{\qbin{2c_n}{n-1-2a_n}}^{-1} \,
\qbin{2a_n}{n-1-2c_n} \,
\qbin{2a_n}{n-1-2b_n}
\\
&=
\dfrac{1}{n^2} \, 
\{2a_n\}!\, \{2b_n\}!\, \{2c_n\}!
\\
&=
\dfrac{(-1)^{n-1}}{n^2} \, 
\left(
\prod_{k=1}^{2a_n} 2\, \sin \frac{k\, \pi}{n}
\right) \, 
\left(
\prod_{k=1}^{2b_n} 2\, \sin \frac{k\, \pi}{n}\right) \, 
\left(\prod_{k=1}^{2c_n} 2\, \sin \frac{k\, \pi}{n}\right) \, 
.  
\end{aligned}
\end{equation*}
\begin{teo}\label{theorem:tetasympt}
The volume of the ideal tetrahedron $T$ with dihedral angles $\alpha$, $\beta$, $\gamma$ is given as follows. 
\begin{equation*}
\operatorname{Vol}(T) = 
\lim_{n\to\infty}
\dfrac{\pi}{n} \, \log\left( (-1)^{n-1} 
\left\{
\begin{matrix}
a_n & b_n & c_n \\
a_n & b_n & c_n
\end{matrix}
\right\}_{tet}\right).
\end{equation*}
Moreover, letting $\overline{a}_n=n-1-a_n, \overline{b}_n=n-1-b_n, \overline{c}_n=n-1-c_n$, it also holds:
\begin{equation*}
\operatorname{Vol}(T) = 
\lim_{n\to\infty}
\dfrac{\pi}{n} \, \log\left( (-1)^{n-1} 
\left\{
\begin{matrix}
\overline{a}_n & \overline{b}_n & \overline{c}_n \\
\overline{a}_n & \overline{b}_n & \overline{c}_n
\end{matrix}
\right\}_{tet}\right).
\end{equation*}
\label{th:ideal}
\end{teo}
\begin{proof}{
The limit is computed as follows.  
$$
\begin{aligned}
&
\lim_{n\to\infty}
\dfrac{ \pi}{n} \, \log (-1)^{n-1}
\left\{
\begin{matrix}
a_n & b_n & c_n \\
a_n & b_n & c_n
\end{matrix}
\right\}_{tet}
\\&
=
\lim_{n\to\infty}
\dfrac{\pi}{n} \, \log 
\dfrac{1}{n^2} \, \left(
\prod_{k=1}^{2a_n} 2\, \sin \frac{k\, \pi}{n} \, 
\prod_{k=1}^{2b_n} 2\, \sin \frac{k\, \pi}{n} \, 
\prod_{k=1}^{2c_n} 2\, \sin \frac{k\, \pi}{n} 
\right) 
\\
&=
\lim_{n\to\infty}
\dfrac{\pi}{n} \,
\left(\!\!
-2 \log n + \sum_{k=1}^{2a_n} \log ( 2\, \sin \frac{k\, \pi}{n})
+\sum_{k=1}^{2b_n} \log ( 2\, \sin \frac{k\, \pi}{n})
+\sum_{k=1}^{2c_n} \log ( 2\, \sin \frac{k\, \pi}{n})
\!\!
\right)
\\
&
=
\Lambda(\alpha) + \Lambda(\beta) + \Lambda(\gamma),
\end{aligned}
$$
which is equal to the volume of $T$.  
Here we use the Lobachevski function $\Lambda(x) = -\displaystyle\int_0^x \log(2\, |\sin x|) \, dx$,  its relation $\Lambda(\pi -x) = -\Lambda(x)$, and
\begin{equation}
\lim_{n\to\infty}
\dfrac{\pi}{n} \, 
\sum_{k=1}^{2a_n} \log( 2\, \sin \frac{k\, \pi}{n})
=
\int_0^{a} \log(|2\, \sin t|) \, dt 
\label{eq:limlog}
\end{equation}
for $0 < a < \pi$ and a sequence of integers $a_n$ such that $\lim_{a \to \infty} \frac{2 \pi a_n }{n} = a$.  
The last statement is a direct consequence of the fact that changing the coloring from $a_n,b_n,c_n$ to $\overline{a}_n,\overline{b}_n,\overline{c}_n$ is equivalent to switching all the orientations of the edges of the tetrahedron in Figure \ref{fig:tet} without changing the labels, and the resulting graph is isotopic to the initial one, thus the invariants are equal for each $n$. 
}
\end{proof}
\subsection{Volume of a truncated tetrahedron}
Let $T$ be a truncated hyperbolic tetrahedron as in Figure \ref{figure:truncated},   
which has four right-angled hexagons and four triangles.  
Let $\theta_a$, $\theta_b $, $\theta_c$, $\theta_d$, $\theta_e$, $\theta_f$ be the dihedral angles at the edges $a$, $b$, $\cdots$, $f$ of $T$.  
Other dihedral angles of $T$ are all right angles.  
The shape of $T$ is uniquely determined by the angles $\theta_a$, $\theta_b $, $\cdots$, $\theta_f$.  
For precise definition of a truncated hyperbolic tetrahedron, see Definition 3.1 in \cite{U}.  
\begin{figure}[htb]
$$
\includegraphics[scale=0.4]{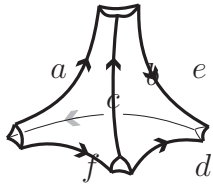}
\raisebox{13mm}{\hspace{-21mm}$a\qquad\qquad e$}
\raisebox{12mm}{\hspace{-8mm}$b$}
\raisebox{9mm}{\hspace{-7mm}${c}$}
\raisebox{0mm}{\hspace{-5mm}$f \qquad\quad d$}
$$
\caption{A truncated hyperbolic tetrahedron $T$.}
\label{figure:truncated}
\end{figure}
\begin{teo}
Let $T$ be the truncated tetrahedron with oriented labeled edges as in Figure \ref{figure:truncated}, and  
let $0 < \theta_a$, $\theta_b $, $\theta_c$, $\theta_d$, $\theta_e$, $\theta_f < \pi$ be the internal dihedral angles at the edges.  
If $\theta_i$, $\theta_j$, $\theta_k$ are three dihedral angles meeting at the same vertex, then they satisfy
$\theta_i + \theta_j + \theta_k < \pi$ since $T$ is a truncated tetrahedron.  
Let $a_n,b_n,c_n,d_n,e_n,f_n$ be sequences of integers such that $\lim_{n\to \infty} \frac{2\pi a_n}{n}=\pi-\theta_a,\ldots, \lim_{n\to \infty} \frac{2\pi f_n}{n}=\pi-\theta_f,$.
Put $\overline{a_n} = n-1-a_n$, $\cdots$, 
$\overline{f_n} = n-1-f_n$.  
Using these parameters, the volume of \,$T$ is given as follows.  
$$
\operatorname{Vol}(T) = 
\lim_{n\to\infty}
\,
\dfrac{\pi}{2\, n}\,
 \log 
 \left(
\left\{
\begin{matrix}
a_n & b_n & e_n \\
d_n & c_n & f_n
\end{matrix}
\right\}_{tet}\,\left\{
\begin{matrix}
\overline a_n & \overline b_n & \overline e_n \\
\overline d_n & \overline c_n & \overline f_n
\end{matrix}
\right\}_{tet}
\right).  
$$
\label{th:trunc}
\end{teo}
\begin{rem}
$
\left\{
\begin{matrix}
a_n & b_n & e_n \\
d_n & c_n & f_n
\end{matrix}
\right\}_{tet}
$
is defined for tetrahedron with oriented edges, and it is not symmetric with respect to the group of symmetries of the tetrahedron.  
However, the limit of the product 
$
\left\{
\begin{matrix}
a_n & b_n & e_n \\
d_n & c_n & f_n
\end{matrix}
\right\}_{tet}
\left\{
\begin{matrix}
\overline a_n & \overline b_n & \overline e_n \\
\overline d_n & \overline c_n & \overline f_n
\end{matrix}
\right\}_{tet}
$
becomes symmetric.  
\end{rem}
\begin{proof}{
The triples $(a_n, b_n, e_n)$, $(a_n, c_n, f_n)$, $(b_n, d_n, f_n)$, $(c_n, d_n, e_n)$, $(\overline a_n, \overline b_n, \overline e_n)$, $(\overline a_n$, $\overline c_n$, $\overline f_n)$, $(\overline b_n, \overline d_n, \overline f_n)$, $(\overline c_n, \overline d_n, \overline e_n)$ are all admissible. 
The range of sum is reduced and the signature is determined as follows. 
\begin{multline}
\left\{
\begin{matrix}
a_n & b_n & e_n \\
d_n & c_n & f_n 
\end{matrix}\right\}_{tet}
\!\!\!\!\!
=
(-1)^{n-1} \!
\dfrac{\{B_{d_ne_nc_n}\}!  \{B_{a_nb_ne_n}\}!}
{\{B_{b_nd_nf_n}\}!  \{B_{a_nf_nc_n}\}!} \!
\qbin{2e_n}{\!\!\!\!A_{a_nb_ne_n}+1-n\!\!\!\!} \!
{\qbin{2e_n}{\!\!\!\!B_{e_nc_nd_n}\!\!\!\!}}^{-1} 
\\
\sum_{z = m_n^{(1)}}^{M_n^{(1)}}\,
\qbin{A_{a_nf_nc_n}+1-n}{2c_n+z+1-n} \, 
\qbin{B_{a_nc_nf_n}+z}{B_{a_nc_nf_n}} \,
\\
\qbin{B_{b_nf_nd_n}+B_{d_ne_nc_n}-z}{B_{b_nf_nd_n}} \,
\qbin{B_{d_nc_ne_n}+z}{B_{d_nf_nb_n}},     
\label{eq:trunc1}
\end{multline}
\begin{multline}
\left\{
\begin{matrix}
\overline a_n & \overline b_n & \overline e_n \\
\overline d_n & \overline c_n & \overline f_n 
\end{matrix}\right\}_{tet}
\!\!\!\!\!\!
=
(-1)^{n-1} \! 
\dfrac
{\{B_{b_nd_nf_n}\}!  \{B_{a_n f_n c_n}\}!}
{\{B_{d_n e_n c_n}\}!  \{B_{a_n b_n e_n}\}!} \!
\qbin{2 e_n}{\!\!\!\!A_{a_n b_n e_n}+1-n\!\!\!\!}^{-1} \!
{\qbin{2 e_n}{\!\!\!\!B_{e_n c_n d_n}\!\!\!\!}} 
\\
\sum_{z = m_n^{(2)}}^{M_n^{(2)}}\,
\qbin{2 c_n-z}{A_{a_n f_n c_n}+1-n} \, 
\qbin{B_{a_n c_n f_n}}{B_{a_n c_n f_n}-z} \,
\\
\qbin{B_{b_n f_n d_n}}{B_{b_n f_n d_n}+B_{d_n e_n c_n}+1-n+z} \,
\qbin{B_{d_n f_n b_n}}{B_{d_n c_n e_n}-z}
,  
\label{eq:trunc2}
\end{multline}
where $m_n^{(1)} = \max(n-1-2c_n,-b_n-c_n+e_n+f_n)$, 
$M_n^{(1)}= \min(B_{d_ne_nc_n}, B_{a_nf_nc_n})$, 
$m_n^{(2)} = \max(0, n-1-b_n+c_n-e_n-f_n)$ and 
$M_n^{(2)}= \min(B_{a_n c_n f_n}, B_{d_n c_n e_n})$.  
Let 
\begin{multline*}
R_n^{(1)}(z)=
\\
\qbin{\!\!\!A_{a_nf_nc_n}+1-n\!\!\!}{\!\!\!2c_n+z+1-n\!\!\!} 
\qbin{\!\!\!B_{a_nc_nf_n}+z\!\!\!}{B_{a_nc_nf_n}} 
\qbin{\!\!\!B_{b_nf_nd_n}+B_{d_ne_nc_n}-z\!\!\!}{B_{b_nf_nd_n}} 
\qbin{\!\!\!B_{d_nc_ne_n}+z\!\!\!}{B_{d_nf_nb_n}},         
\end{multline*}
for $m^{(1)} < z < M^{(1)}$ and
\begin{multline*}
R_n^{(2)}(z)=
\\
\qbin{2 c_n-z}{\!\!\!A_{a_n f_n c_n}\!+1-n\!\!\!\!} \!\!
\qbin{B_{a_n c_n f_n}}{\!\!\!B_{a_n c_n f_n}\!-z\!\!\!\!}\!\!
\qbin{B_{b_n f_n d_n}}{\!\!\!B_{b_n f_n d_n}\!+B_{d_n e_n c_n}\!\!+1-n+z\!\!\!\!} \!\!
\qbin{B_{d_n f_n b_n}}{\!\!\!B_{d_n c_n e_n}\!-z\!\!\!\!} 
\end{multline*}
for $m^{(2)} < z < M^{(2)}$.  
Let 
$$
r_n^{(1)}(z) = \dfrac{R_n^{(1)}(z)}{R_n^{(1)}(z-1)},  
\qquad 
r_n^{(2)}(z) = \dfrac{R_n^{(2)}(z)}{R_n^{(2)}(z-1)}.
$$  
Then 
\begin{multline*}
r_n^{(1)}(z) 
=\\
-\dfrac{\{B_{a_nf_nc_n}-z+1\} \, \{B_{a_nc_nf_n}+z\}\,
\{B_{d_ne_nc_n}-z+1\} \, \{B_{d_nc_ne_n}+z\}}
{\{2c_n+z+1\} \, \{z\} \, 
 \{ B_{b_nf_nd_n}+B_{d_ne_nc_n}-z+1\} \,
\{B_{d_nc_ne_n} - B_{d_nf_nb_n} + z\} }, 
\end{multline*}
\begin{multline*}
r_n^{(2)}(z) 
=\\
-\dfrac{\{B_{a_nf_nc_n}+z\} \, \{B_{a_nc_nf_n}-z+1\}\,
\{B_{d_ne_nc_n}+z\} \, \{B_{d_nc_ne_n}-z+1\}}
{\{2c_n-z+1\} \, \{z\} \, 
 \{ B_{b_nf_nd_n}+B_{d_ne_nc_n}+1+z\} \,
\{B_{d_nf_nb_n} - B_{d_nc_ne_n} +z\} }.     
\end{multline*}
According to Lemma \ref{lemma:trunc0}, 
$R_n^{(1)}$ and $R_n^{(2)}$ are positive, and there are unique integers $z_n^{(1)}$ and  $z_n^{(2)}$ 
where $z_n^{(1)}$ coresponds
to the maximum of $R_n^{(1)}(z)$  
and 
$z_n^{(2)}$ corresponds to the maximum of $R_n^{(2)}(z)$ at the integer points. 
Then we have
$$
R_n^{(1)}(z_n^{(1)}) \, R_n^{(2)}(z_n^{(2)})
\leq
\left\{
\begin{matrix}
a_n & b_n & e_n \\
d_n & c_n & f_n
\end{matrix}
\right\}_{tet}\left\{
\begin{matrix}
\overline a_n & \overline b_n & \overline e_n \\
\overline d_n & \overline c_n & \overline f_n
\end{matrix}
\right\}_{tet}
\!\!\!\!
\leq
n \, R_n^{(1)}(z_n^{(1)}) \, R_n^{(2)}(z_n^{(2)}).  
$$
Hence
\begin{multline}
\lim_{n\to\infty}
\,
\dfrac{\pi}{2\, n}\,
 \log 
 \left(
\left\{
\begin{matrix}
a_n & b_n & e_n \\
d_n & c_n & f_n
\end{matrix}
\right\}_{tet}\,\left\{
\begin{matrix}
\overline a_n & \overline b_n & \overline e_n \\
\overline d_n & \overline c_n & \overline f_n
\end{matrix}
\right\}_{tet}
\right)
=
\\
\lim_{n\to\infty}
\,
\dfrac{\pi}{2\, n}\,
\log\left(
R_n^{(1)}(z_n^{(1)}) \, R_n^{(2)}(z_n^{(2)})
\right).
\label{eq:lim}
\end{multline}
Let  
$
\zeta^{(1)} = \lim_{n\to\infty} \frac{2 \, \pi}{n} \, z_n^{(1)}$,
$
\zeta^{(2)} = \lim_{n\to\infty} \frac{2 \, \pi}{n} \, z_n^{(2)}
$.
For $0 < \beta <  \alpha < \pi$, 
we have
$$
\lim_{n\to\infty}
\dfrac{\pi}{n} \, 
\log\qbin{\lfloor\frac{\alpha \, n}{\pi}\rfloor}{\lfloor\frac {\beta \, n}{\pi}\rfloor}
=
-\Lambda(\alpha)+\Lambda(\beta) + \Lambda(\alpha-\beta)
$$
by the definition of the Riemann integral.  
Replacing every quantum binomial in $R_n^{(1)}$ and $R_n^{(2)}$ in \eqref{eq:lim} by using this relation, we get
\begin{multline}
\lim_{n\to\infty}
\dfrac{\pi}{n} \, \log
\left\{
\begin{matrix}
a_n & b_n & e_n \\
d_n & c_n & f_n 
\end{matrix}\right\}_{tet} \, 
\left\{
\begin{matrix}
\overline a_n & \overline b_n & \overline e_n \\
\overline d_n & \overline c_n & \overline f_n 
\end{matrix}\right\}_{tet}
=
\\
g(\theta_a, \theta_b, \theta_c, \theta_d, \theta_e, \theta_f, \zeta ^{(1)})
-
g(\theta_a, \theta_b, \theta_c, \theta_d, \theta_e, \theta_f, -\zeta ^{(2)}), 
\label{eq:limit}
\end{multline}
where
\begin{multline}
g(\theta_a, \theta_b, \theta_c, \theta_d, \theta_e, \theta_f, \zeta)
=
\\
\Lambda(\frac{\pi - \theta_a - \theta_f + \theta_c-\zeta }{2})
+
\Lambda(\frac{- 2\, \theta_c + \zeta}{2})
-
\Lambda(\frac{\pi-\theta_a-\theta_c+\theta_f+ \zeta}{2})
\\
+
\Lambda(\frac{\zeta}{2})
-
\Lambda(\frac{- \theta_b + \theta_c - \theta_e - \theta_f-\zeta}{2}
 )
+
\Lambda(\frac{\pi - \theta_d - \theta_e + \theta_c - \zeta}{2})
\\
-
\Lambda(\frac{\pi - \theta_d - \theta_c + \theta_e + \zeta}{2})
+
\Lambda(\frac{-\theta_b - \theta_c + \theta_e + \theta_f + \zeta}{2}).   
\end{multline}
Moreover,
 $\zeta^{(1)}$ and $-\zeta^{(2)}$ are solutions of 
\begin{equation}
\dfrac{d}{d\zeta} \,
g(\theta_a, \theta_b, \theta_c, \theta_d, \theta_e, \theta_f, \zeta)
=0
\label{equation:zeta}
\end{equation}  
satisfying 
 $$
 \begin{aligned}
\max(2 \, \theta_c,\theta_b+\theta_c-\theta_e-\theta_f) &< \zeta ^{(1)} <
\min\left({\pi + \theta_c - \theta_d-\theta_e}, \ 
{\pi - \theta_a + \theta_c - \theta_f}\right), 
\\[8pt]
\max\left(0,\ {\theta_b-\theta_c + \theta_e + \theta_f}\right) &< \zeta ^{(2)} < 
\min\left({\pi - \theta_a - \theta_c + \theta_f},\  
{\pi - \theta_c- \theta_d  + \theta_e}\right),  
\end{aligned}
$$ 
since
$$
\begin{aligned}
\lim_{n\to\infty}\dfrac{\pi}{n}
\log r_n^{(1)}(z_n)
&=
\dfrac{d}{d\zeta} \,
g(\theta_a, \theta_b, \theta_c, \theta_d, \theta_e, \theta_f, \zeta),
\\
\lim_{n\to\infty}\dfrac{\pi}{n}
\log r_n^{(2)}(z_n)
&=
-\dfrac{d}{d\zeta} \,
g(\theta_a, \theta_b, \theta_c, \theta_d, \theta_e, \theta_f, -\zeta),
\end{aligned}
$$
where $z_n$ is a sequence such that $\lim_{n\to\infty}\frac{2 \pi}{n} z_n = \zeta$.  
Taking the exponential of two times the both sides of \eqref{equation:zeta},
we get
\begin{equation}
\dfrac{|\cos(\frac{\theta_a -\theta_c + \theta_f + \zeta}{2})|  \,
|\cos(\frac{\theta_a + \theta_c - \theta_f - \zeta}{2})|  \,
|\cos(\frac{\theta_c - \theta_d - \theta_e - \zeta}{2})|\,
|\cos(\frac{\theta_c + \theta_d - \theta_e - \zeta}{2})|}
{|\sin(\frac{\zeta-2 \theta_c}{2})|   \,
| \sin(\frac{-\theta_b + \theta_c - \theta_e - \theta_f - \zeta}{2})|\,
|\sin(\frac{-\theta_b - \theta_c + \theta_e + \theta_f + \zeta}{2})| \,
|\sin(\frac{\zeta}{2})|}
=1.  
\label{equation:abs}
\end{equation}
If $\zeta = \zeta^{(1)}$, $\sin(\frac{-\theta_b + \theta_c - \theta_e - \theta_f - \zeta^{(1)}}{2})$ is negative and the other seven values of the trigonometric functions in \eqref{equation:abs}  are all positive.
If $\zeta= -\zeta^{(2)}$, 
$\sin(\frac{-\zeta^{(2)}-2 \theta_c}{2})$,
$\sin(\frac{-\zeta^{(2)}}{2})$, 
$\sin(\frac{-\theta_b - \theta_c + \theta_e + \theta_f - \zeta^{(2)}}{2})$
are negative and the other five values of the trigonometric functions in \eqref{equation:abs} are positive.  
Theorefore, $\zeta^{(1)}$ and $-\zeta^{(2)}$ are solutions of 
$$
\dfrac{\cos(\frac{\theta_a -\theta_c + \theta_f + \zeta}{2}) \, \cos(\frac{\theta_a + \theta_c - \theta_f - \zeta}{2})  \, \cos(\frac{\theta_c - \theta_d - \theta_e - \zeta}{2})\,\cos(\frac{\theta_c + \theta_d - \theta_e - \zeta}{2})}
{\sin(\frac{\zeta-2 \theta_c}{2})    \sin(\frac{-\theta_b + \theta_c - \theta_e - \theta_f - \zeta}{2})\sin(\frac{-\theta_b - \theta_c + \theta_e + \theta_f + \zeta}{2}) \sin(\frac{\zeta}{2})}
= -1, 
$$
and it is equivalent to a quadratic equation 
\begin{equation}
C_2 \, z^2 + C_1 \, z + C_0 = 0
\label{equation:quadratic}
\end{equation}
where $z = e^{ i \zeta}$, 
$$
\begin{aligned}
C_0 &=
a\,b\,c^2\,d + a\,b\,c^4\,d + 
  a\,b\,c^3\,e + 
  a\,b\,c^3\,d^2\,e + 
  \\&\qquad\qquad
  b\,c^3\,d\,f + 
  a^2\,b\,c^3\,d\,f + 
  a\,c^3\,d\,e\,f + 
  a\,b^2\,c^3\,d\,e\,f,
  \\
C_1 &=
-a\,b\,c^2\,d + 
  a\,b\,c^2\,d\,e^2 + 
  b\,c^2\,e\,f + 
  a^2\,b\,c^2\,e\,f - 
  a\,c\,d\,e\,f - 
  a\,b^2\,c\,d\,e\,f - 
     \\&\quad
 a\,c^3\,d\,e\,f - 
  a\,b^2\,c^3\,d\,e\,f + 
  b\,c^2\,d^2\,e\,f + 
  a^2\,b\,c^2\,d^2\,e\,f + 
  a\,b\,c^2\,d\,f^2 - 
  a\,b\,c^2\,d\,e^2\,f^2,
\\
C_2 &=
a\,c\,d\,e\,f + 
  a\,b^2\,c\,d\,e\,f + 
  b\,c\,d\,e^2\,f + 
  a^2\,b\,c\,d\,e^2\,f + 
  \\&\qquad\qquad
  a\,b\,c\,e\,f^2 + 
  a\,b\,c\,d^2\,e\,f^2 + 
  a\,b\,d\,e^2\,f^2 + 
  a\,b\,c^2\,d\,e^2\,f^2,
\end{aligned}
$$
and $a = e^{i \theta_a}$, $\cdots$, $f = e^{i\theta_f}$.    
The two solutions  $z_1 = e^{i\zeta^{(1)}}$, $z_2 = e^{-i\zeta^{(2)}}$  of \eqref{equation:quadratic} are given by
$$
z_1  =
\dfrac{-C_1 - 4 \,a \, b \, c^2 \, d \, e \, f \sqrt{\det G}}{2 \, C_2},
\qquad
z_2  =
\dfrac{-C_1 + 4 \,a \, b \, c^2 \, d \, e \, f \sqrt{\det G}}{2 \, C_2},  
$$
where $G$ is the Gram matrix  of $T$ given by
\begin{equation}
G = 
\begin{pmatrix}
1 & -\cos \theta_a & - \cos \theta_b & -\cos \theta_f \\
-\cos\theta_a & 1 & -\cos\theta_e & -\cos\theta_c \\
-\cos\theta_b & -\cos\theta_e & 1 & -\cos\theta_d \\
-\cos\theta_f & -\cos\theta_c & -\cos\theta_d & 1
\end{pmatrix}.    
\label{eq:gram}
\end{equation}
since
$
\det G = \frac{C_1^2 - C_0 \, C_2}{16 (a \,b \,c^2 \,d \,e \,f)^2}.
$  
Recall that the determinant $\det G$ is a negative real number since $T$ is a truncated hyperbolic tetrahedron,  and we assign $\sqrt{\det G} = i \, \sqrt{-\det G}$.    
\par
To compare \eqref{eq:limit} with the volume of $T$, we use
the method in \cite{U} based on the following Schl\"afli's differential formula:   
\begin{equation}
d\, \operatorname{Vol}(T)
=
-\dfrac{1}{2}\,
\left(l_a \, d\theta_a + l_b \, d\theta_b + l_c \, d\theta_c + l_d\, d\theta_d + l_e \, d\theta_e + l_f \, d\theta_f\right),   
\label{eq:diff}
\end{equation}
where $l_a$, $\cdots$, $l_f$ are lengths of edges labeled by $a$, $\cdots$, $f$ respectively.  
Let $G_{ij}$ denote the submatrix of $G$ obtained by deleting the $i$-th row and $j$-th column, and let $c_{ij}= (-1)^{i+j}\, \det G_{ij}$ be the corresponding cofactor.   
Then, by the formula (5.2) in \cite{U}, the length $l_a$ is given by
\begin{equation}
2\, l_a =
\log\left(
\dfrac{2 \, c_{34}^2 - c_{33}\, c_{44} + 2\, c_{34} \, \sqrt{-\det G} \, \sin \theta_a}{c_{33}\, c_{44}}
\right).  
\label{eq:length}
\end{equation}
On the other hand, 
\begin{multline}
2\,\dfrac{\partial}{\partial \theta_a} 
 \left(g(\theta_a, \theta_b, \theta_c, \theta_d, \theta_e, \theta_f, \zeta ^{(1)})
-
g(\theta_a, \theta_b, \theta_c, \theta_d, \theta_e, \theta_f, -\zeta ^{(2)})\right)
=
\\
\log\Big(\dfrac{\cos(
 \frac{ \theta_a - \theta_c+ \theta_f + \zeta ^{(1)}}{2}
) \, 
 \cos(
 \frac{\theta_a + \theta_c - \theta_f + \zeta ^{(2)}}{2}
)
 }
{
 \cos(
 \frac{\theta_a - \theta_c + \theta_f - \zeta ^{(2)}}{2}
 )
  \cos(
 \frac{\theta_a + \theta_c - \theta_f - \zeta ^{(1)}}{2}
 )
} \Big).  
\label{eq:diffg}
\end{multline}
To rationalize the denominator of the right-hand side of \eqref{eq:diffg}, we compute
\begin{multline}
\dfrac{\cos(
 \frac{ \theta_a - \theta_c+ \theta_f + \zeta ^{(1)}}{2}
) \, 
 \cos(
 \frac{\theta_a + \theta_c - \theta_f + \zeta ^{(2)}}{2}
)
 }
{
 \cos(
 \frac{\theta_a - \theta_c + \theta_f - \zeta ^{(2)}}{2}
 )
  \cos(
 \frac{\theta_a + \theta_c - \theta_f - \zeta ^{(1)}}{2}
 )
}
=
\dfrac{(a\, f\, z_1 + c)\,(a \,c + f\, z_2 )
 }
{
(a \,f \,z_2 + c )\,(a\, c + f z_1 )
}
\\
=
\dfrac{(a\, f\, z_1 + c)^2\,(a \,c + f\, z_2 )^2
 }
{
(a\, f\, z_1 + c)\,(a \,f \,z_2 + c )\,(a\, c + f z_1 )\,(a\, c + f\,  z_2)
}, 
\label{equation:computation}
\end{multline}
where $a=e^{i\theta_a}$, $\cdots$, $f = e^{i\theta_f}$, $z_1 = e^{i\zeta^{(1)}}$, $z_2= e^{-i\zeta^{(2)}}$ as before.  
Then \eqref{equation:computation} turns out to be equal to 
$$
\frac{2  c_{34}^2 - c_{33}c_{44} - 2 c_{34}  \sqrt{-\det G} \, \sin \theta_a}{c_{33}\, c_{44}}
=
\left(
\frac{2 c_{34}^2 - c_{33}c_{44} + 2 c_{34}  \sqrt{-\det G} \, \sin \theta_a}{c_{33}\, c_{44}}
\right)^{-1}
$$
by an actual computation.  
Hence we get
$$
\dfrac{\partial}{\partial \theta_a} \, 
 \big(g(\theta_a, \theta_b, \theta_c, \theta_d, \theta_e, \theta_f, \zeta ^{(1)})
-
g(\theta_a, \theta_b, \theta_c, \theta_d, \theta_e, \theta_f, -\zeta ^{(2)})\big)
=
-l_a.  
$$  
\par
To see the differential with respect to $\theta_b$, we use the symmetry given in \eqref{equation:symmetry}, 
$$
\left\{
\begin{matrix}
a & b & e \\
d & c & f
\end{matrix}
\right\}_{tet} 
=
\left\{
\begin{matrix}
b & \overline e & \overline a \\
c & f & d
\end{matrix}
\right\}_{tet}.  
$$  
Please note that the triples $(b, \overline e, \overline a)$, $(b, f, d)$, $(c, \overline e, \overline a)$, $(c, f, \overline a)$ are all admissible.  
Applying the above argument and the fact that $\cos \theta = \cos(-\theta)$, we have 
$$
\dfrac{\partial}{\partial \theta_b} \, 
g(\theta_b, -\theta_e, \theta_f, \theta_c, -\theta_a, \theta_d, \zeta_b^{(1)})
-
g(\theta_b, -\theta_e, \theta_f, \theta_c, -\theta_a, \theta_d, -\zeta_b^{(2)})
= -l_b, 
$$
where $\zeta_b^{(1)}$ and $-\zeta_b^{(2)}$ are solutions of the equation
$$
\dfrac{\partial}{\partial \zeta}
\,  g(\theta_b, -\theta_e, \theta_f, \theta_c, -\theta_a, \theta_d, \zeta)
= 0
$$
satisfying
$$
\max(2 \, \theta_f, \ \theta_e + \theta_f + \theta_a - \theta_d)
< \zeta_b^{(1)} <
{\pi - \theta_b - \theta_d + \theta_f},  
$$
$$
\max(0, \ \theta_e - \theta_f - \theta_a + \theta_d) < \zeta_b^{(2)} < 
\pi-\theta_f - \theta_c - \theta_a. 
$$ 
This implies that
$$
\dfrac{\partial}{\partial \theta_b} \, 
 (g(\theta_a, \theta_b, \theta_c, \theta_d, \theta_e, \theta_f, \zeta^{(1)})
-
g(\theta_a, \theta_b, \theta_c, \theta_d, \theta_e, \theta_f, -\zeta^{(2)}))
=
-l_b.  
$$  
Similarly, for $\theta_c$, $\theta_d$, $\theta_e$, $\theta_f$, we have
$$
\begin{matrix}
\dfrac{\partial}{\partial \theta_c} \, 
 (g(\theta_a, \theta_b, \theta_c, \theta_d, \theta_e, \theta_f, \zeta^{(1)})
-
g(\theta_a, \theta_b, \theta_c, \theta_d, \theta_e, \theta_f, -\zeta^{(2)}))
=
-l_c,  
\\
\vdots
\\
\dfrac{\partial}{\partial \theta_f} \, 
 (g(\theta_a, \theta_b, \theta_c, \theta_d, \theta_e, \theta_f, \zeta^{(1)})
-
g(\theta_a, \theta_b, \theta_c, \theta_d, \theta_e, \theta_f, -\zeta^{(2)}))
=
-l_f. 
\end{matrix}
$$  
Hence there is a constant ${C}$ such that
$$
g(\theta_a, \theta_b, \theta_c, \theta_d, \theta_e, \theta_f, \zeta ^{(1)})
-
g(\theta_a, \theta_b, \theta_c, \theta_d, \theta_e, \theta_f, -\zeta ^{(2)})
=
2\, \operatorname{Vol}(T) + C.  
$$
But $g(\theta_a, \theta_b, \theta_c, \theta_d, \theta_e, \theta_f, \zeta ^{(1)}) - g(\theta_a, \theta_b, \theta_c, \theta_d, \theta_e, \theta_f, -\zeta ^{(2)})$ can be extended continuously to the case of ideal tetrahedra whose dihedral angles satisfy $\theta_a+\theta_b + \theta_e=\theta_a+\theta_c + \theta_f=\theta_b+\theta_d + \theta_f=\theta_c+\theta_d + \theta_e=\pi$, and Theorem \ref{th:ideal} implies that the constant $C=0$.    
Therefore, Theorem \ref{th:trunc} holds.  
}
\end{proof}

\section{Invariants of graphs}
In this section we exploit the graphical relations satisfied by the algebraic objects studied in the preceding sections to define an invariant of framed, oriented, colored trivalent graphs. Then, we provide a face-model computing the invariant. 
\subsection{Graphical relations between symbols}
\begin{prop}
Letting $t_a = a \, (a+1-n)$ and using the above explained graphical convention to draw morphisms of representations, the following relations hold:
%
\begin{equation}
\begin{matrix}
a \\
\begin{picture}(30,50)
\thicklines
\put(0,60){\line(0,-1){20}}
\put(0,60){\vector(0,-1){15}}
\put(0,40){\line(1,-1){10}}
\put(15,25){\line(1,-1){10}}
\put(25,15){\line(0,1){25}}
\put(25,40){\line(-1,-1){25}}
\put(0,15){\vector(0,-1){15}}
\end{picture}
\end{matrix}
=
\xi_n^{2t_a}\  \,  
\begin{matrix}
a \\
\begin{picture}(13,50)
\thicklines
\put(0,60){\vector(0,-1){60}}
\end{picture}
\end{matrix},
\qquad\quad
\begin{matrix}
a \quad\qquad b
\\
\begin{picture}(30,65)(0,-5)
\thicklines
\put(30,60){\line(-1,-1){30}}
\put(30,30){\vector(-1,-1){15}}
\put(0,60){\line(1,-1){10}}
\put(20,40){\line(1,-1){10}}
\put(0,30){\vector(1,-1){15}}
\put(15,15){\vector(0,-1){20}}
\end{picture}
\\
c
\end{matrix}
=
\xi_n^{t_c-t_a-t_b} \!\!\!\!\! 
\begin{matrix}
a \qquad\qquad b
\\
\begin{picture}(50,60)
\thicklines
\put(0,60){\vector(1,-1){25}}
\put(50,60){\vector(-1,-1){25}}
\put(25,35){\vector(0,-1){35}}
\end{picture}
\\
c
\end{matrix},
\label{eq:t}
\end{equation}
\begin{equation}
\begin{matrix}
a \qquad b \ \ c
\\
\begin{picture}(45,65)(0,-5)
\thicklines
\put(0,60){\vector(1,-2){15}}
\put(30,60){\line(0,-1){10}}
\put(30,40){\vector(0,-1){25}}
\put(15,30){\vector(1,-1){15}}
\put(45,60){\vector(-1,-1){30}}
\put(30,15){\vector(0,-1){20}}
\put(13,13){$e$}
\end{picture}
\\
\quad d
\end{matrix}
=
\sum_{f}
\xi_n^{t_a+t_b-t_e-t_f} \,  
\left\{\begin{matrix}
c & a & e \\
b & d & f
\end{matrix}\right\}_{\xi_n} \!\!\!\!
\begin{matrix}
a\qquad b\qquad c
\\
\begin{picture}(53,65)(0,-5)
\thicklines
\put(0,60){\vector(1,-2){15}}
\put(30,60){\vector(-1,-2){15}}
\put(15,30){\vector(1,-1){15}}
\put(52.5,60){\vector(-1,-2){22.5}}
\put(30,15){\vector(0,-1){20}}
\put(13,13){$f$}
\end{picture}
\\
\ \ d
\end{matrix},
\label{eq:c6j}
\end{equation}
\begin{equation}
\qquad\qquad
\begin{matrix}
a \qquad b \ \ c
\\
\begin{picture}(45,65)(0,-5)
\thicklines
\put(0,60){\vector(1,-2){15}}
\put(30,60){\vector(0,-1){45}}
\put(15,30){\vector(1,-1){15}}
\put(45,60){\line(-1,-1){10}}
\put(25,40){\vector(-1,-1){10}}
\put(30,15){\vector(0,-1){20}}
\put(13,13){$e$}
\end{picture}
\\
\quad d
\end{matrix}
=
\sum_{f}
\xi_n^{-t_a-t_b+t_e+t_f} \,  
\left\{\begin{matrix}
c & a & e \\
b & d & f
\end{matrix}\right\}_{\xi_n} \!\!\!\!
\begin{matrix}
a\qquad b\qquad c
\\
\begin{picture}(53,65)(0,-5)
\thicklines
\put(0,60){\vector(1,-2){15}}
\put(30,60){\vector(-1,-2){15}}
\put(15,30){\vector(1,-1){15}}
\put(52.5,60){\vector(-1,-2){22.5}}
\put(30,15){\vector(0,-1){20}}
\put(13,13){$f$}
\end{picture}
\\
\ \ d
\end{matrix}.
\label{eq:c6jm}
\end{equation}
\end{prop}
\bigskip\par
\begin{proof}{
The first relation of \eqref{eq:t} is easily proved by applying the operators ${}_a^aR$, $\cup_{a,n-1-a}$ and $\cap_{a,n-1-a}$ to the vector $e_0^a\in V^a$.
The second relation of \eqref{eq:t} is obtained by applying $Y_c^{b,a}$ and ${}_b^a R$ to the highest weight vector $e_0^c$ in $V^c$.  
Let $u = a+b-c$.  
By using \eqref{equation:R} and \eqref{equation:QCGC}, we know that
The coefficient of 
$e_{0}^a \otimes e_{u}^b$ in ${}_b^a R\big(Y_c^{b,a}(e_0^c)\big)$ is 
$$
\sqrt{-1}^{-u}\,\xi_n^{2(b-u)a-\frac{u(2b-u+1)}{2}} \, \qbin{2c}{a+b+c-(n-1)},
$$ 
while the coefficient of 
$e_{0}^a \otimes e_{u}^b$ in $Y_c^{a,b}(e_0^c)\big)$ is 
$$
\sqrt{-1}^{-u} \, (-1)^u \, \xi_n^{\frac{u(2b-u+1)}{2}}\, \qbin{2c}{a+b+c-(n-1)}.  
$$
Hence ${}_b^a R\circ Y_c^{b,a} = \xi_n^{t_c-t_a-t_b} \, Y_c^{a,b}$.  
\par
Relation \eqref{eq:c6j} is proved by using the deformation of the diagram given in Figure~\ref{figure:deformation} and applying twice relation \eqref{eq:t} and the definition of the $6j$-symbols.  
The graphical meaning of the $6j$-symbol is given in Figure \ref{fig:6j}  and we also use \eqref{eq:t}.
Relation  \eqref{eq:c6jm} is proved similarly.    
}
\end{proof}
\begin{figure}[htb]
$$
\includegraphics[scale=0.25]{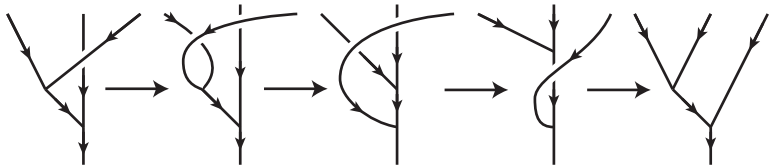}
$$
\caption{The sequence of deformations}
\label{figure:deformation}
\end{figure}
\begin{rem}
As representations of ${\mathcal U}_{\xi_n}(sl_2)$, $V^a$ and $V^{a+2n}$ are isomorphic, but the  twist operators are different since $a \notin \frac{1}{2}\mathbb Z$.
Such difference comes from the factor $q^{\frac{H\otimes H}{2}}$ in the universal $R$-matrix.  
\end{rem}
\subsection{Construction}
Let now $\Gamma$ be a framed oriented trivalent graph in $S^3$ and let us fix once and for all a natural number $n\geq 2$ as well as a root 
$\xi_n=\exp(\frac{\pi\sqrt{-1}}{n})$. 
Let us first fix the notation we shall use in this section: let $E_0,\ldots, E_r$ be the edges of $\Gamma$ and let us assume that the framing of $\Gamma$ forms an orientable surface; this enables us to always assume that, in our drawings, the framing of $\Gamma$ lies horizontally above the blackboard plane (indeed this can be always achieved up to modifying the diagram of $\Gamma$, see \cite{C3} Lemma 2.3).
\begin{defi}[Coloring]
A coloring on $\Gamma$ is a map $\operatorname{col}:\{\text{edges}\}\to \mc\setminus \frac{1}{2}\mz$ such that for each three-tuple of edges $E_i$, $E_j$, $E_k$ sharing a vertex $v$ (possibly two edges coincide) it holds: 
$$ 
f_v(E_i)+f_v(E_j)+f_v(E_k)\in \{n-1,n,\ldots, 2n-2\}, 
$$
where $f_v(E_i)$ is $\operatorname{col}(E_i)$ if $v$ is the end of $E_i$ and $n-1-\operatorname{col}(E_i)$ otherwise. 
\end{defi}
\par
Given a trivalent graph $\Gamma$ embedded in $S^3$ equipped with an orientation of its edges, a framing and a coloring (such a datum is called 
\emph{colored oriented graph}), 
we can associate to it a complex number which we shall denote $<\Gamma,\ \operatorname{col}>_n$ by the following
 construction.
 \begin{enumerate}
 \item Choose an edge $E_0$ of $\Gamma$ and cut $\Gamma$ open along $E_0$.
 \item Move by an isotopy $\Gamma$ so to put it in a $(1,1)$-tangle diagram and so that the two open strands (initially contained in $E_0$) are directed towards the bottom. 
 \item Assigning $\cap$ operators to the maximal points, $\cup$ operators to the minimal points, $R$-matrices to the crossing points and the Clebsch-Gordan operators $Y_a^{b,c}$ and $Y_{a,b}^c$ to the trivalent vertices as in \cite{KR}, 
we associate to the  diagram $D$ of $\Gamma$ obtained in (2) an operator 
$\operatorname{op}(D):V^{\operatorname{col}(E_0)}\to V^{\operatorname{col}(E_0)}$ 
and hence, by Schur's lemma, a scalar $\lambda(D)\in \mc$. 
 \item Define the scalar associated to $D$ as 
 $i(D)= \lambda(D)\qbin{2\operatorname{col}(E_0)+n}{2\operatorname{col}(E_0)+1}^{-1}$.
 \end{enumerate}
\begin{teo}\label{teo:invariance}
The scalar $i(D)$ is independent on all the choices of the above construction and is therefore an invariant $<\Gamma,\operatorname{col}>_n\in \mc$ of the colored oriented graph embedded in $S^3$.
\end{teo}
\begin{proof}{
Cut open $G$ along an edge $E_0$; the fact that $\lambda(D)$ (and hence $i(D)$) is an invariant up to isotopy of the colored framed $(1,1)$-tangle represented by $D$ is a standard consequence of the properties of representations of quantum groups and in particular of ${\mathcal U}_{\xi_n}(sl_2)$. 
So we need to prove that cutting $\Gamma$ open along a different edge, say $E_1$ and repeating the construction we get the same invariant.
\par
So let us cut $\Gamma$ open along $E_0$ and $E_1$ and, up to isotopy, put the result in a position of a $(2,2)$-tangle whose boundary strands are oriented towards the bottom and are included one in $E_0$ and one in $E_1$ both at the top and at the bottom. 
Let $D$ be a diagram of $\Gamma$ in such a position, 
$\operatorname{col}(E_0)=a$ and $\operatorname{col}(E_1)=b$. 
The operator represented by $D$ is 
$\operatorname{op}(D):V^{a}\otimes V^{b}\to V^{a}\otimes V^{b}$ 
and by Clebsch-Gordan decomposition and Schur's lemma there exist scalars $h_{a+b-k}(D)\in \mc, k\in \{0,1,\ldots, n-1\}$ such that $\operatorname{op}(D)$ restricted to the submodule $V^{a+b-k}$ of $V^{a}\otimes V^{b}$ is $h_{a+b-k}(D)\, \operatorname{id}$.
Let us call $i(D_0)$ and $i(D_1)$ the invariant computed out of $D$ by closing $E_1$ (i.e. cutting open $\Gamma$ along $E_0$) and closing $E_0$ (i.e. cutting $\Gamma$ open along $E_1$) respectively. Then it clearly holds: 
$$
i(D_0)=\sum_{l=a+b-n+1}^{l=a+b}h_l(D)\, i\big(\raisebox{-0.8cm}{
\psfrag{a}{$a$}\psfrag{b}{$b$}\psfrag{c}{$l$}
\includegraphics[width=0.8cm]{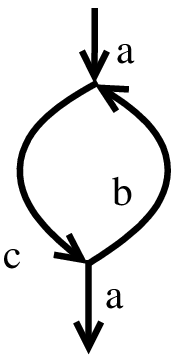}}\big), \ \ 
i(D_1)=\sum_{l=a+b-n+1}^{l=a+b}h_l(D)\,  i\big(\raisebox{-0.8cm}{
\psfrag{a}{$b$}\psfrag{b}{$l$}\psfrag{c}{$a$}
\includegraphics[width=0.8cm]{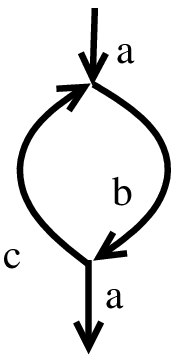}}\big).
$$
But then $i(D_0)=i(D_1)$ as a consequence of Lemma $\ref{lem:thetagraph}$.
}
\end{proof}
\begin{rem}
\begin{enumerate}
\item
We have
\begin{equation}
i\big(\raisebox{-0.5cm}{
\psfrag{a}{$a$}\psfrag{b}{$b$}\psfrag{c}{$c$}
\includegraphics[width=1.2cm]{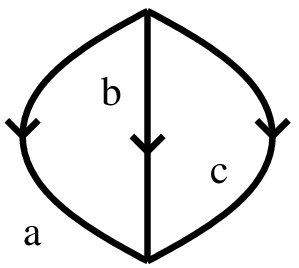}}\big)=1.  
\label{equation:theta}
\end{equation}
\item
Colored graphs include colored links, and for links the invariant defined above coincides with the colored Alexander invariant given in \cite{ADO} and discussed in \cite{GR}, \cite{M}.  
Lemma \ref{lem:thetagraph} gives a new proof for the independence of the string to cut to make a (1, 1)-tangle.  
This was first proved in \cite{ADO} by computation, and then refined and extended to more general settings in \cite{GR} by means of theoretical arguments.  
Comparing with the proof in \cite{GR}, we see that
$
\qbin{2a + n}{2a+1}^{-1}
$
corresponds to
$
d(a) 
$ 
in Definition 2.3 of \cite{GR} expressing the ``virtual degree'' of the representation $V^a$.  
\end{enumerate}
\end{rem}

\subsection{Face model for the invariant}
In this subsection, we construct a face model for the invariant $<\Gamma,\operatorname{col}>_n$ which in particular include the colored Alexander invariant defined in \cite{M} by the method already used in Section 6 of \cite{KR}. 
 
\par
\begin{figure}[htb]
$$
\begin{matrix}
\includegraphics[scale=0.25]{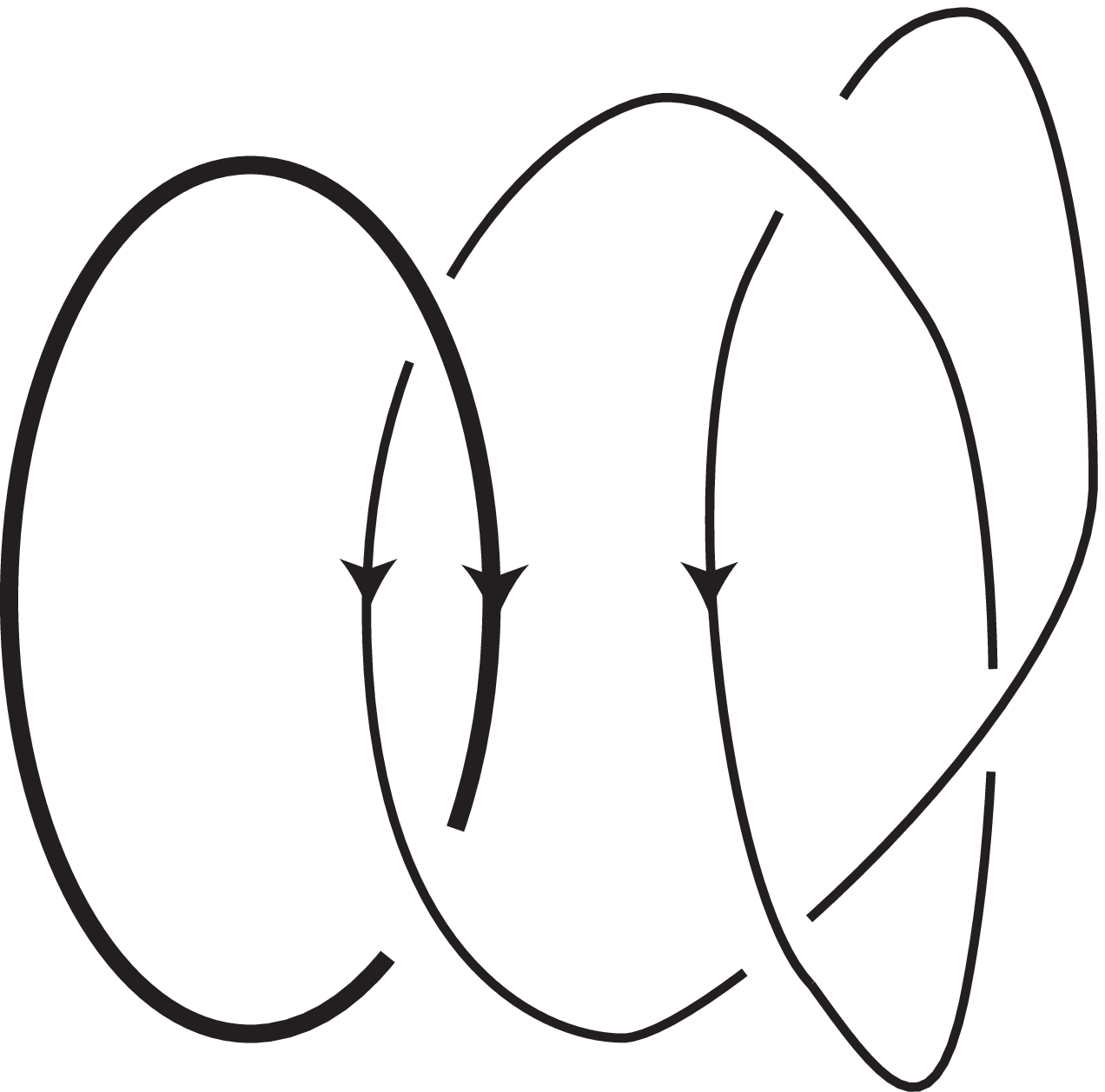}
\\
K_1 \qquad K_2
\\
L
\end{matrix}
\quad
\longrightarrow
\quad
\begin{matrix}
\includegraphics[scale=0.2]{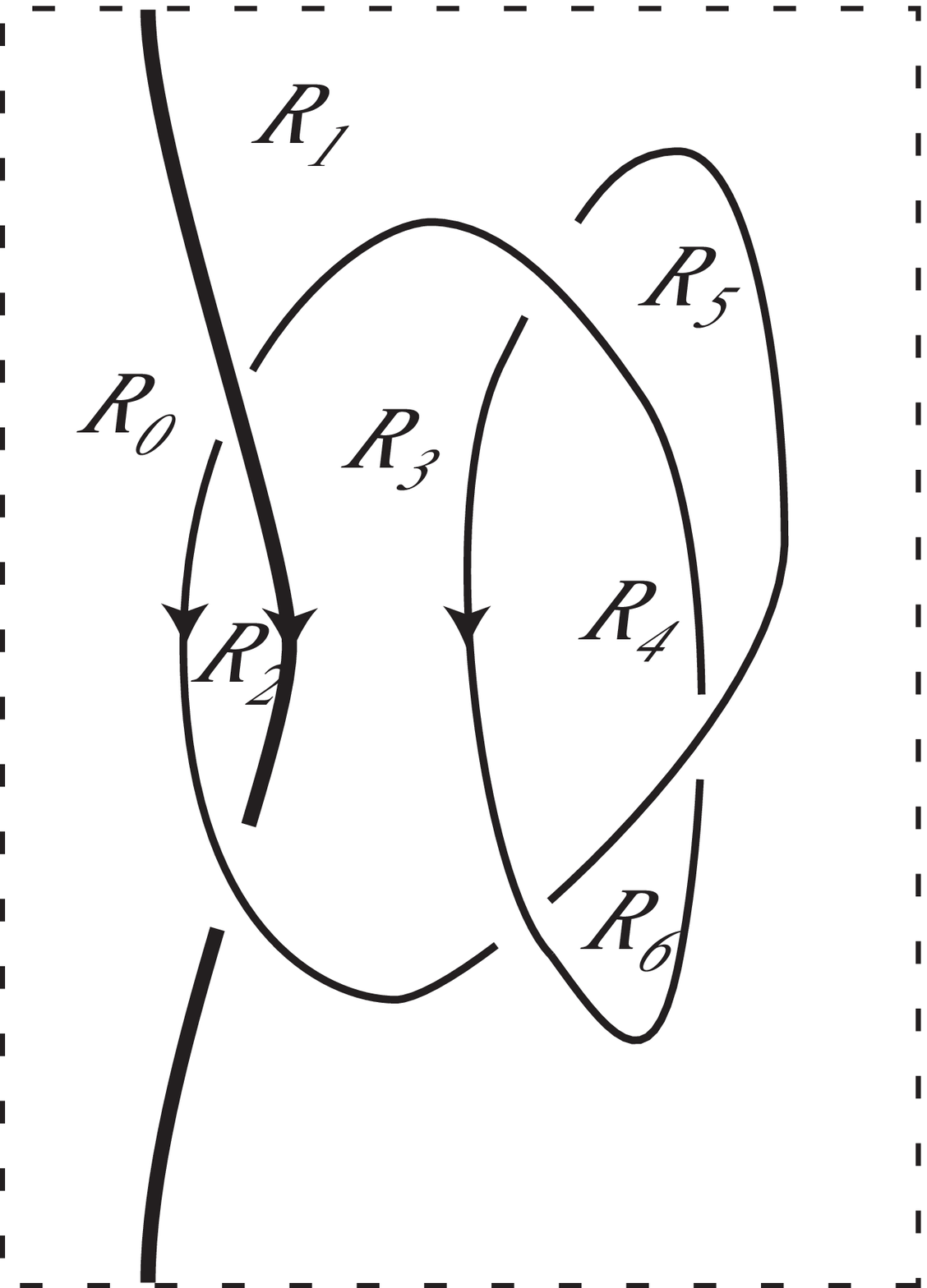}
\\
T_L
\end{matrix}
$$
\caption{A tangle diagram $T_L$ related to the link $L$ and its regions: the colors of the edges are $\lambda_i$ and if an arc is directed upwards the color is counted as $n-1-\lambda_i$.}
\label{figure:tangle}
\end{figure}
\par
Let $(\Gamma,\operatorname{col})$ be a colored, oriented, framed trivalent graph. Let us cut it open along the edge $E_1$ and put it in a $(1,1)$-tangle like position so that the two strands which were contained in $E_1$ are directed towards the bottom. The diagram $T_\Gamma$ just constructed splits the plane into \emph{regions} $R_0,\ldots, R_k$, where we let $R_0$ and $R_1$ be respectively the leftmost and the rightmost regions.
Let $a_0$, $a_1\in \mathbb{C}\setminus\frac{1}{2}\mathbb{Z}$ be complex numbers satisfying $a_0+ \operatorname{col}(E_1)-a_1\in \{0,1,\ldots, n-1\}$. 
We define a \emph{state} of $T_\Gamma$ as a mapping
$
\varphi : \{R_0,R_1, R_2, \cdots, R_d\}
\longrightarrow {\mathbb C}
$
which satisfies the following conditions.  
\begin{enumerate}
\item
\ $\varphi(R_0) = a_0$, $\varphi(R_1) = a_1$, 
\item
\ If $R_i$ and $R_j$ are adjacent along $e_k$,  $R_i$ is on the left of $e_k$ and $R_j$ is on the right of $e_k$, then 
$\varphi(R_i) + \operatorname{col}(e_k)$ (or $\varphi(R_i) + n-1-\operatorname{col}(e_k)$ if $e_k$ is oriented upwards) $ = \varphi(R_j) + l$, where $l$ is an integer and $0 \leq l \leq n-1$.   
\item
\ $\varphi(R_i)$ is not a half-integer for any $i$.  
\end{enumerate}
Note that the third condition is a condition for $a_0$. 
Let $Z_{a_0,a_1}(T_\Gamma)$ be the following state sum.
\begin{multline}
Z_{a_0, a_1}(T_\Gamma)
=
\!\!
\sum_{\begin{matrix}
\scriptstyle \varphi :\\
\scriptstyle\text{states}
\end{matrix}
}
\prod_{p\, :\, \text{maximum}}\!\!\!\!\!\!W_{\text{max}}(p)
\prod_{p\, :\, \text{minimum}}\!\!\!\!\!\!W_{\text{min}}(p)
\prod_{p\, :\, \text{crossing}}\!\!\!\!\!\!W_{\text{c}}(p)
\prod_{p\, :\, \text{vertex}}\!\!\!\!\!\!W_{\text{v}}(p)
\label{eq:statesum}
\end{multline}
where $W_{\text{max}}(p)$, $W_{\text{min}}(p)$, $W_{\text{crossing}}(p)$, $W_{\text{vertex}}(p)$ are given as follows (in the pictures we denote by $\lambda,\mu, \eta$ the colors of the edges of $\Gamma$ and by $a,b,c$ the states of the regions).  
\begin{equation*}
\setlength{\unitlength}{0.7mm}
\raisebox{-4mm}{$a$} \ \lambda \,
\begin{matrix}
p \\
\begin{picture}(20,15)
\thicklines
\put(10,5){\oval(20,20)[t]}
\put(0,5){\vector(0,-1){5}}
\put(20,5){\line(0,-1){5}}
\put(7, 3){$b$}
\end{picture}
\end{matrix}\ %
\underset{W_{\max}(p)}{\longrightarrow}\ %
\qbin{2a + n}{2a+1},  
%
\qquad
\setlength{\unitlength}{0.7mm}
\raisebox{4mm}{$a$}  \ \,
\begin{matrix}
\begin{picture}(20,15)
\thicklines
\put(10,10){\oval(20,20)[b]}
\put(0,10){\line(0,1){5}}
\put(20,15){\vector(0,-1){5}}
\put(7, 7){$b$}
\end{picture}
\\ p
\end{matrix}
\lambda
\ %
\underset{W_{\min}(p)}{\longrightarrow}\ %
\qbin{2b+n}{2b+1}^{-1}.    
\end{equation*}
(The left one corresponds to \eqref{eq:theta} and
the right one corresponds to \eqref{eq:decomposition}.)  
\begin{equation*}
\setlength{\unitlength}{0.5mm}
\begin{matrix}
\lambda \qquad\quad \mu
\\
\begin{picture}(30,30)
\thicklines
\put(30,30){\vector(-1,-1){30}}
\put(0,30){\line(1,-1){10}}
\put(20, 10){\vector(1, -1){10}}
\put(12,23){$d$}
\put(2,13){$a$}
\put(23,13){$b$}
\put(12,2){$c$}
\end{picture}
\end{matrix}
\ \ 
\underset{W_{\text{c}}(p) }{\longrightarrow}
\begin{matrix}
\xi_n^{t_a+t_b-t_c-t_d} 
\left\{\begin{matrix}
\mu & a & c \\
\lambda & b & d
\end{matrix}
\right\}_{\xi_n},
\end{matrix} 
\end{equation*}
\begin{equation*}
\quad
\setlength{\unitlength}{0.5mm}
\begin{matrix}
\lambda \qquad\quad \mu
\\
\begin{picture}(30,30)
\thicklines
\put(0,30){\vector(1,-1){30}}
\put(30,30){\line(-1,-1){10}}
\put(10, 10){\vector(-1, -1){10}}
\put(12,23){$d$}
\put(2,13){$a$}
\put(23,13){$b$}
\put(12,2){$c$}
\end{picture}
\end{matrix}
\ \ 
\underset{W_{\text{c}}(p) }{\longrightarrow}
\begin{matrix}
\xi_n^{-t_a-t_b+t_c+t_d} 
\left\{\begin{matrix}
\mu & a & c \\
\lambda & b & d
\end{matrix}
\right\}_{\xi_n}.
\end{matrix} 
\end{equation*}
\begin{equation*}
\setlength{\unitlength}{0.5mm}
\begin{matrix}
\lambda \qquad\quad \mu
\\
\begin{picture}(30,30)
\thicklines
\put(0,30){\vector(1,-1){15}}
\put(30,30){\vector(-1,-1){15}}
\put(15,15){\vector(0,-1){15}}
\put(14,28){$a$}
\put(2,15){$b$}
\put(28,15){$c$}
\end{picture}\\
\eta
\end{matrix}
\ \ 
\underset{W_{\text{v}}(p) }{\longrightarrow}
\left\{\begin{matrix}
\mu & \lambda & \eta \\
b & c & a
\end{matrix}
\right\}_{\xi_n},
\qquad
%
\setlength{\unitlength}{0.5mm}
\begin{matrix}
\eta
\\
\begin{picture}(30,30)
\thicklines
\put(15,15){\vector(1,-1){15}}
\put(15,15){\vector(-1,-1){15}}
\put(15,30){\vector(0,-1){15}}
\put(14,0){$a$}
\put(2,15){$b$}
\put(28,15){$c$}
\end{picture}\\
\lambda \qquad\quad \mu
\end{matrix}
\ \ 
\underset{W_{\text{v}}(p) }{\longrightarrow}
\qbin{2\eta+n}{2\eta+1}\, 
\left\{\begin{matrix}
b & \lambda & a \\
\mu & c & \eta
\end{matrix}
\right\}_{\xi_n}.
\end{equation*}
(The first two correspond to \eqref{eq:c6j} and \eqref{eq:c6jm}.) 

\begin{teo}[Face model for $<\Gamma,\operatorname{col}>_n$]\label{teo:facemodel}
Let 
$$
\widetilde Z_{a_0, a}(T_\Gamma) = 
\qbin{2 \operatorname{col}(E_1)+n}{2 \operatorname{col}(E_1) + 1}^{-1} \, 
Z_{a_0, a_1}(T_\Gamma).
$$  
Then $\widetilde Z_{a_0, a_1}(T_\Gamma)=<\Gamma,\operatorname{col}>_n$. 
In particular, if $\Gamma$ is a link, then 
$$
\dfrac{1}{n\,\sqrt{-1}^{\, n-1}} \, \widetilde Z_{a_0, a_1}(T_\Gamma)$$
is equal to the 
colored Alexander invariant in \cite{M}.  
\end{teo}
\begin{figure}[hbt]
\qquad\ %
$\includegraphics[scale=0.22]{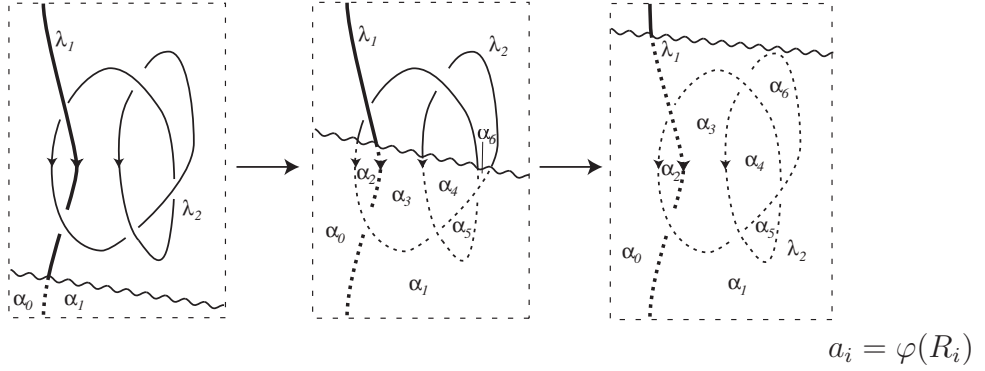}$
\hfill
\\
\hfill$a_i = \varphi(R_i)$
\caption{Vertex-IRF correspondence in \cite{KR}}
\label{figure:vertexIRF}
\end{figure}
\par
\begin{proof}{
The first statement can be proved as in Section 6 of \cite{KR} about the correspondence of vertex models and IRF models (see also \cite{C3}), we sketch here the basic ideas (see Figure \ref{figure:vertexIRF}).  
\par
As explained previously, the diagram $T_\Gamma$ induces an operator 
$$
\op(T_\Gamma):V^{\operatorname{col}(E_1)}\to V^{\operatorname{col}(E_1)}
$$ 
and hence a scalar (which we previously called $\lambda(T_{\Gamma})$). 
Fix a height function $t:T_\Gamma \to [0,1]$ and let now $\ell_t$ be the horizontal at height $t$. 
Intersecting $T_{\Gamma}$ with a generic $\ell_t$, we get an ordered sequence of colors $\operatorname{col}(E_{i_1})$, $\operatorname{col}(E_{i_2})$, $\ldots$, $\operatorname{col}(E_{i_k})$ (where $E_{i_1},\ldots, E_{i_k}$ are the edges of $T_{\Gamma}$ intersecting $\ell_t$, read from left to right); if moreover a state $s$ is fixed on $T_{\Gamma}$ then there is an induced sequence of colors $c_0,\ldots, c_k$ of the segments composing $\ell\setminus \ell\cap T_\Gamma$, such that in particular $c_0=a_0$ and $c_k=a_1$; we will denote this sequence $s\vert_{\ell_t}$. 
Therefore, for each line $\ell$ and state $s$ we can consider the morphism $\op(\ell,s\vert_\ell):V^{a_1}\to V^{a_0}\otimes V^{\operatorname{col}(E_1)}$ defined as follows. 
First define a map $\op_{bottom}(\ell,s\vert_\ell):V^{a_1}\to V^{a_0}\otimes V^{\operatorname{col}(E_{i_1})}\cdots V^{\operatorname{col}(E_{i_k})}$ as the composition $Y^{c_0,\col(E_{i_1})}_{c_1}\circ\cdots\circ Y^{c_{k-2},\col(E_{i_{k-1}})}_{c_{k-1}} \circ Y^{c_{k-1},\col(E_{i_k})}_{a_1}$.  
Then let $\op_{top}(\ell,s\vert_\ell):V^{\operatorname{col}(E_{i_1})}\cdots V^{\operatorname{col}(E_{i_k})}\to V^{\operatorname{col}(E_1)}$ be defined as explained in the previous section using the part of $T_\Gamma$ lying above $\ell$.
Then let $\op(\ell,s\vert_\ell)= (id_{a_0}\otimes \op_{top}(\ell,s\vert_\ell))\circ \op_{bottom}(\ell,s\vert_\ell)$.
\par
To prove the theorem it is sufficient to prove that if $t_1$ and $t_2$ are two levels such that the preimage of the interval $[t_1,t_2]$ by the height function contains exactly only one extremum or trivalent vertex or crossing of $T_\Gamma$, then $\op(\ell_{t_1},s\vert_{\ell_{t_1}})=\sum_{s'} c(s')\op(\ell_{t_2},s'\vert_{\ell_{t_2}})$ where $s'$ ranges over all the states $s'$ of the subdiagram of $T_\Gamma$ lying between $\ell_{t_1}$ and $\ell_{t_2}$ and such that $s'\vert_{\ell_{t_1}}=s\vert_{\ell_{t_1}}$.
This is sufficient because following these equalities while $t$ goes from $0$ to $1$ we get $\op(\ell_0,s\vert_{\ell_0})=Z_{a_0,a_1}(T_\Gamma)\, \op(\ell_1,s\vert_{\ell_1})$ 
where both $s\vert_{\ell_0}$ and $s\vert_{\ell_1}$ are the sequence $(a_0,a_1)$ and by construction 
$\op(\ell_0,s\vert_{\ell_0})
=
id_{a_0}\otimes (\lambda(T_\Gamma) \operatorname{Id}_{\operatorname{col}(E_1)})\circ Y_{a_1}^{a_0,\operatorname{col}(E_1)}$ 
and 
$\op(\ell_1,s\vert_{\ell_1})=Y^{a_0,\operatorname{col}(E_1)}_{a_1}$, 
so that $Z_{a_0,a_1}(T_\Gamma)=\lambda(T_\Gamma)$.
\par
So to conclude, the reader can check that the coefficients $c(s)$ associated to maxima are those computed in Lemma \ref{lem:thetagraph}, those associated to minima are computed in equation \ref{eq:decomposition}, those associated to a vertex with one leg on the bottom of the picture are $6j$-symbols (by the definition of $6j$-symbols), those associated to a vertex with two legs on the bottom come from equality expressed in Figure \ref{figure:another6j}, and finally those associated to crossings come from equations \eqref{eq:c6j} and \eqref{eq:c6jm}.

To prove the last statement let us note that the left diagram  of Figure \ref{figure:vertexIRF} represents the scalar operator $O_{T_L}^n(\lambda_1, \cdots, \lambda_r) : V_{\lambda_1} \to V_{\lambda_1}$ defined by the vertex model in \cite{M}.  
On the other hand, the right diagram of Figure \ref{figure:vertexIRF}
represents $Z_{a_0, a_1}(T_L)$.  
Hence  $Z_{a_0, a_1}(T_L) = O_{T_L}^n(\lambda_1, \cdots, \lambda_r)$  and, by comparing with the definition  in \cite{M}, 
$\frac{\widetilde Z_{a_0, a_1}(T_L)}{\sqrt{-1}^{\,n-1}\, n}$ 
is equal to the colored Alexander invariant since 
$\{n-1\}! = \sqrt{-1}^{\,n-1}\, n$.   
}
 \end{proof}
 %
\par\noindent
\begin{rem}
\begin{enumerate}
\item
(Independence of the choice of $a_0$ and $a_1$)
As a corollary of Theorem \ref{teo:facemodel} the scalar $\tilde{Z}_{a_0, a_1}(T_L)$ does not depend on $a_0$ and $a_1$. 
\item
(Face model for the Kashaev invariant)
We constructed the above face model assuming that no color is a half integer.  
If we consider the case of a link and let its colors $\lambda_i$ tend to $(n-1)/2$, we have 
$$
\lim_{\lambda\to(n-1)/2}
\qbin{2\lambda+n}{2\lambda+1}=
(-1)^{n-1}
$$ 
and, by using this limit, we know that
$
\lim_{\lambda_1, \cdots, \lambda_r\to\frac{n-1}{2}}
(-1)^{n-1} \, \tilde Z_{a_0, a_1}(T_L)
$
is equal to the Kashaev invariant.  
Therefore,  as long as all the colors of the regions in each state are not half-integers (and this is satisfied for a generic choice of $a_0$ and $a_1$) every term in the state sum is well-defined, while the  face model introduced in \cite{KR} is not well-defined if the color of some component is $\frac{n-1}{2}$.  
\item
(Face model for the classical Alexander invariant)
The case $n=2$ corresponds to the classical Alexander polynomial and the Conway potential function for knots and links.  
For these invariants, a face model was already constructed in \cite{V} by using the representation theory of quantum supergroup ${gl(1|1)}$ from another point of view.    
\end{enumerate}
\label{rem:kashaev}
\end{rem}
\appendix
\section{Proofs of algebraic statements}
\subsection{Proof of Theorem \ref{teo:ClebschGordan}}\label{app:Clebsch-Gordan}
The coefficient of $e_r^a \otimes e_p^b$ in
$$
\Delta(E) \, 
\left(
\sum\limits_{u+v=a+b-c} C^{a,b,c}_{u,v,0}\, e^a_u\otimes e^b_v\right)
$$
is $[r+1]  \xi_n^{b-p}  C_{r+1, p, 0}^{a, b, c}+
[p+1]  \xi_n^{-a+r}  C_{r,p+1, 0}^{a, b, c}$ which is equal to:
\begin{multline*}
(-1)^p \, \sqrt{-1}^{c-a-b} \, \xi_n^{\frac{(p+1)(2b-p) -( r+1)(2a-r)}{2}}  
\qquad\qquad\qquad\qquad\qquad
\\
\qquad\qquad\qquad
\qbin{2c}{a+b+c-(n-1)}\,\qbin{a+b-c}{r} \, 
\left(
[a+b-c-r]
-
[p+1]
\right).  
\label{eq:deltae}
\end{multline*}
This is $0$ because $a + b - c = r + p + 1$.
The coefficient of $e_r^a \otimes e_p^b$ in
$$
\Delta(F) \, 
\left(
\sum\limits_{u+v-t=a+b-c} C^{a,b,c}_{u,v,t}\, e^a_u\otimes e^b_v
\right)
$$
is $[2a-r+1]  \xi_n^{b-p}  C_{r-1, p, t}^{a, b, c}+
[2b-p+1]  \xi_n^{-a+r} C_{r,p-1, t}^{a, b, c}$, which after shifting the summation indices in  Formula \eqref{equation:QCGC} becomes:
\begin{multline*}
(-1)^{p-t} \, \sqrt{-1}^{c-a-b} \, \xi_n^{\frac{p(2b-p+1) -r(2a-r+1)}{2}} \,
\qbin{2c}{2c-t}^{-1} \, \qbin{2c}{a+b+c-n+1}\, 
\\
\Big(
[2a-r+1] \, \xi_n^{{a+b-r-p+1}}
\qquad\qquad\qquad\qquad
\qquad\qquad\qquad\qquad
\\
\qquad
\sum_{z+w=t+1,\ z\geq 1}
(-1)^{z-1} 
\xi_n^\frac{(2z-t-2)(2c-t+1)}{2} 
\left[\begin{matrix}
a+b-c \\ r-z
\end{matrix}
\right] \!\!
\left[\begin{matrix}
2a-r+z \\ 2a-r+1
\end{matrix} \right]\!\!
\left[\begin{matrix}
2b-p+w \\ 2b-p
\end{matrix} 
\right]
\\
-
[2b-p+1]  \xi_n^{-(a+b-r-p+1)}
\qquad\qquad\qquad\qquad
\qquad\qquad\qquad\qquad
\\
\qquad
\sum_{z+w=t+1,\  w\geq 1} \!\!\!\!\!\!\!\!
(-1)^z \, 
\xi_n^\frac{(2z-t)(2c-t+1)}{2} 
\left[\begin{matrix}
a+b-c \\ r-z
\end{matrix}
\right]\!\!\!
\left[\begin{matrix}
2a-r+z \\ 2a-r
\end{matrix} \right]\!\!\!
\left[\begin{matrix}
2b-p+w \\ 2b-p+1
\end{matrix} 
\right] \!\!
\Big)
\end{multline*}
\begin{multline*}
=
(-1)^{p-t-1} \sqrt{-1}^{c-a-b}  \xi_n^{\frac{p(2b-p+1) -r(2a-r+1)}{2}} 
\qbin{2c}{2c-t}^{-1} \qbin{2c}{a+b+c-n+1}
\\
\sum_{z+w=t+1}\, 
(-1)^z \, 
\xi_n^\frac{(2z-t)(2c-t+1)}{2} \,
\left[\begin{matrix}
a+b-c \\ r-z
\end{matrix}
\right] \,
\left[\begin{matrix}
2a-r+z \\ 2a-r
\end{matrix} \right]\,
\left[\begin{matrix}
2b-p+w \\ 2b-p
\end{matrix} 
\right]
\\
\left(
\xi_n^{-(a+b-r-p+1)}\,
{[w]} 
+
\xi_n^{{a+b-r-p-2c+t}}\,
{[z]}
\right)
\end{multline*}
which is easily seen to be equal to $[2c-t]C^{a,b,c}_{r,p,t+1}$.
Hence we have:
$$\Delta(F) \left(\sum_{u+v=a+b-c+t} C^{a,b,c}_{u,v,t} e^a_u\otimes e^b_v
\right)
=
[2c-t] 
\sum_{r+p = a+b-c+t+1}C_{r,p,t+1}^{a, b, c}
e_r^a \otimes e_p^b.  
$$
These relations imply that the subspace of 
$V^a \otimes V^b$ 
spanned by $$
\sum_{u+v=a+b-c} \!\!\!\! C^{a,b,c}_{u,v,0}\, e^a_u\otimes e^b_v,
\sum_{u+v=a+b-c+1}  \!\!\!\! C^{a,b,c}_{u,v,1}\, e^a_u\otimes e^b_v,
\cdots,\ %
\sum_{u+v=a+b-c+n-1}  \!\!\!\! C^{a,b,c}_{u,v,n-1}\, e^a_u\otimes e^b_v
$$
is isomorphic to the highest weight module with the highest weight $c$ and $Y_c^{a,b}$ is a ${\mathcal U}_{\xi_n}(sl_2)$ module map.  
The unicity statement is a consequence of Schur's lemma and of Proposition \ref{prop:decomposition}.\qed
\subsection{Proof of Lemma \ref{lemma:bend}}\label{app:bend}
It is sufficient to prove the identity for a single three tuple $u,v,t$ because the space of morphisms from 
$V^a\otimes V^b$ to $V^c$ is at most one dimensional (this is a consequence of Proposition \ref{prop:decomposition} and of the fact that $V^c$ and $V^{c+i}$ are never isomorphic if $i\in \{0$, $\cdots$, $2n-1\}$).
\par
Therefore let us fix $u=0$, $v=a+b-c$ and $t=0$.
Then the above equality becomes
\begin{multline}
\sqrt{-1}^{b+a-c-n+1} (-1)^{a+b-c} \xi_n^{-a(n-1)} 
\xi_n^{\frac{-(n-1)(n-2a)}{2}} 
\!\!\!\!\! \!\!
\sum_{z+w=a+b-c} 
\!\!\!\!\! \!\!
(-1)^{z} 
\xi_n^{\frac{(2z-a-b+c)(b+c-a+1)}{2}} 
\\
\qbin{n-1+c-a-b}{c-a-b+z}
\qbin{n-1-2a+z}{n-1-2a}
\qbin{c+a+b-z}{2c}
=
\\[12pt]
\sqrt{-1}^{a+b-c-n+1}(-1)^{n-1-a-b+c}
\xi_n^{(n-1)(c-a)}
\xi_n^{\frac{(n-1-a-b+c)(n-1+a-b-c+1)}{2}}\!
\qbin{2a}{\!\!\!\!c+a-b\!\!\!\!}.  
\end{multline}
But the summands on the left hand side are non-zero only for $z = a+b-c$, therefore, 
after some simplification the equality reduces to:
$$
\qbin{n-1-a+b-c}{n-1-2a}
=
\qbin{2a}{c+a-b}.  
$$

\qed

\subsection{Proof of Lemma \ref{lem:thetagraph}}\label{app:thetaval}
By Schur's lemma the diagram represents a multiple of $Id_a$: thus it is sufficient to compute its action on $e^a_0$. Using Proposition \ref{prop:Yprojectors}, we need to prove:
 $$
 \qbin{2a+n}{2a+1}
 =
 \sum_{t=0}^{n-1+c-b-a} 
 C^{b, n-1-c,n-1-a}_{a+b-c+t,n-1-t, n-1}\,
  C^{c, n-1-b, a}_{t,n-1+c-a-b-t,0}.  
 $$
 By Lemma \ref{lemma:bend} it holds $C^{b,n-1-c,n-1-a}_{n-1-v,n-1-t,n-1}=\xi_n^{(n-1)(n-1-b-v-a)}C^{n-1-c,a,n-1-b}_{n-1-t,0,v}$; moreover noting that if $v=0$ in formula \eqref{equation:QCGC}, the sum reduces to the only term with $z=t$, we have: $$
\begin{aligned}
&\sum_{t+v = c+n-1-b-a}
C_{t,v,0}^{c, n-1-b,a}\xi_n^{(n-1)(n-1-b-v-a)}C^{n-1-c,a,n-1-b}_{n-1-t,0,v}
\\[12pt]&
=
\sum_{t+v = c+n-1-b-a}\!\!\!\!\!\!\!\!\!\!
\Big(\sqrt{-1}^{a+b-c-n+1}(-1)^{v}
\\&
\qquad\qquad\qquad\qquad
\xi_n^{\frac{v(2(n-1-b)-v+1)-t(2c-t+1)}{2}}
\qbin{2a}{\!\! a+c-b\!\!} \!\!
\qbin{\!\! c+n-1-b-a\!\!}{t}
\!\Big)
\\&
\Big(\sqrt{-1}^{c-a-b}\xi_n^{(n-1)(n-1-b-v-a)}
\xi_n^{\frac{(t-n+1)(n-2c+t)}{2}}\qbin{2n-2-2b}{2n-2-2b-v}^{-1}
\\&\qquad\qquad\quad
\qbin{2n-2-2b}{n-1-c-b+a}\,
\xi_n^{\frac{v(2(n-1)-2b-v+1)}{2}}
\qbin{2(n-1)-c-a-b}{n-1-2c+t}\Big)
\end{aligned}
$$
Simplifying the above formula and applying to the last binomial the third of the identities recalled at the beginning of Subsection \ref{sub:1}, this reduces to:
$$
\begin{aligned}
&
=
({-1})^{a+b-c} \, 
\xi_n^{(n-1-a-b+c)(n+a-b-c)} \, 
\qbin{2a}{a-b+c}
\\
&
\qquad\sum_{t=0}^{n-1-a-b+c}
\xi_n^{-2(a+1)t} \, 
\qbin{2c-t}{a+b+c-n+1} \, \qbin{n-1+a-b-c+t}{n-1+a-b-c} \,
\\[12pt]
&=
({-1})^{a+b-c} \, 
\qbin{2a}{a-b+c}
\qbin{a-b+c+n}{2a+1}
=
\qbin{2a+n}{2a+1}
\end{aligned}
$$
where the first equality is proved by using the relation \eqref{eq:c} given in Lemma \ref{lemma:relation}.   \qed

\subsection{\bf Proof of Theorem \ref{theo:6j}}\label{app:theo6j}
Using the value of the theta graph \eqref{eq:theta}, 
we have the relation in Figure \ref{figure:another6j}.   
This gives the following expression of 
the quantum $6j$-symbol.  
\begin{figure}[htb]
$$
\begin{matrix}
a\qquad\quad f \\
\hspace{-6mm}
\includegraphics[scale=0.4]{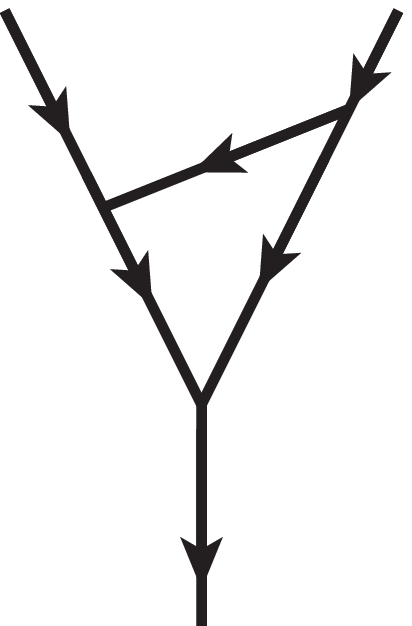}
{\hskip-5mm c}
\end{matrix}
\hspace{-15mm} e
\hspace{4mm}\raisebox{8mm}{\it{b}}
\hspace{3mm} d
\ %
=
\left\{
\begin{matrix}
a & b & e \\
d & c & f
\end{matrix}
\right\}_{\xi_n}
\!\!\!\!
\begin{matrix}
a \qquad\quad f \\
\hspace{-4mm}
\includegraphics[scale=0.4]{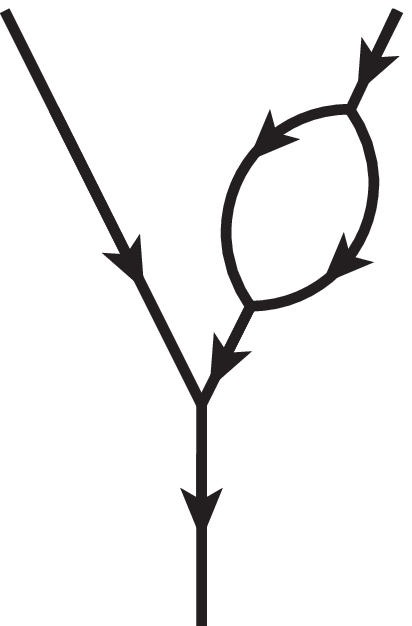}
\hspace{-1cm}\raisebox{2cm}{{\it b}}
\hspace{8mm}\raisebox{15mm}{{\it d}}
\hspace{-8mm}\raisebox{10mm}{{\it f}}
\hspace{7mm}{}
{\hskip-8mm c}
\end{matrix}
\quad
=
{\qbin{2f+n}{2f+1}} \, 
\left\{
\begin{matrix}
a & b & e \\
d & c & f
\end{matrix}
\right\}_{\xi_n}
\!\!\!\!\!\!\!\!
\begin{matrix}
a \qquad\qquad f \\
\includegraphics[scale=0.4]{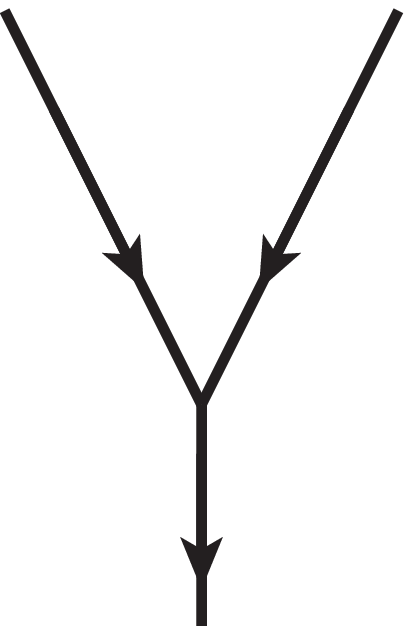}
{\hskip-5mm c}
\end{matrix}.
$$
\caption{Another expression of the quantum $6j$-symbol} 
\label{figure:another6j}
\end{figure}
\begin{multline}
\qbin{2f+n}{2f+1}  
\left\{
\begin{matrix}
a& b & e \\
d & c & f
\end{matrix}
\right\}_{\xi_n}
=
\\
\left(C_{m_2,m_1,m_3}^{a, f, c}\right)^{-1} 
\!\!\!\!
\sum_{m_4, m_5, m_6}
\!\!\!\!
C_{n-1-m_4,n-1-m_5,n-1-m_1}^{n-1-d,n-1-b,n-1-f}
C_{m_2,m_5,m_6}^{a,b,e} 
C_{m_6,m_4,m_3}^{e,d,c}.
\end{multline}
In the above formula we used Proposition \ref{prop:Yprojectors}.
Let us put $m_1 = 0$, $m_3 = 0$, $m_4 = \alpha$, 
then $m_2 = a+f-c$, $m_5 = b+d-f-\alpha$, $m_6 = e+d-c-\alpha$.  
Using \eqref{equation:QCGC}, letting $\{x, x-h\} = \prod_{k=0}^{h-1} \{x-k\}$, and applying Lemma \ref{lemma:bend} to $C_{n-1-m_4,n-1-m_5,n-1-m_1}^{n-1-d,n-1-b,n-1-f}$, we compute the quantum $6j$-symbol as follows.  
In the computation, $n_1=n-1$.  
\begin{equation*}
\qbin{2f+n}{2f+1} 
\left\{
\begin{matrix}
a & b & e \\
d & c & f
\end{matrix}
\right\}_{\xi_n}
\qquad\qquad\qquad\qquad\qquad\qquad\qquad\qquad\qquad\qquad\qquad\quad
\end{equation*}
\begin{equation*}
\begin{aligned}
&=
\left(C_{m_2,m_1,m_3}^{a, f, c}\right)^{-1}
\!\!\!\!\!
\sum_{m_4, m_5, m_6}\!\!\!\!
(\xi_n^{-(f-d+\alpha)(n-1)}\, C_{n-1-m_5,m_1,m_4}^{n-1-b,f,d}) \,
C_{m_6,m_4,m_3}^{e,d,c} \,
C_{m_2,m_5,m_6}^{a,b,e} 
\end{aligned}
\end{equation*}
\begin{equation*}
\begin{aligned}
&=
\Big( \sqrt{-1}^{-B_{afc}}\xi_n^{-\frac{B_{afc}}{2}(B_{acf}+1)}\qbin{2c}{\!\!\!A_{afc}-n_1\!\!\!}\Big)^{-1}
\sum_{\alpha=0}^{B_{dec}}\Big(\xi_n^{-n_1(f-d+\alpha)}\sqrt{-1}^{B_{bdf}-n_1}
\\&
\xi_n^{\frac{(B_{bdf}-n_1-\alpha)(n-\alpha-B_{bfd})}{2}}
\!\!\!\!\!\!\!
\dfrac{\{\alpha\}!}{\{2d,2d-\alpha\}}
\dfrac{\{2d,B_{dfb}\}}{\{B_{bdf}\}!}
\xi_n^{\frac{\alpha(2d-\alpha+1)}{2}}
\dfrac{\{n_1-B_{bfd},n_1-B_{bfd}-\alpha\}}{\{\alpha\}!}
\!\Big)
\\&
\Big(\!\sqrt{-1}^{-B_{edc}}(-1)^{\alpha}\xi_n^{\frac{\alpha(2d-\alpha+1)-(B_{edc}-\alpha)(B_{ecd}+\alpha+1))}{2}}
\dfrac{\{2c,A_{edc}-n_1\}\{B_{edc}\}!}{\{n_1-B_{edc}\}!\{B_{edc}-\alpha\}!\{\alpha\}!}\Big)
\\&
\Big(\sqrt{-1}^{-B_{abe}}(-1)^{b+c-e-f}
\xi_n^{\frac{(B_{bdf}-\alpha)(B_{bfd}+\alpha+1)-B_{afc}(B_{acf}+1)}{2}} 
\\&
\qquad\qquad
\dfrac{\{B_{edc}-\alpha\}!}{\{2e,B_{ecd}+\alpha\}}
\dfrac{\{2e,A_{abe}-n_1\}}{\{n_1-B_{abe}\}!}
\sum_{\substack{z+w=B_{dec}-\alpha,\\ B_{dec}-B_{bdf}\leq z\leq B_{afc}}}
\!\!\!\!\!\!\!\!\!\!\!\!
 (-1)^{z}\xi_n^{\frac{(z-w)(B_{ecd}+\alpha+1)}{2}}
\\&
\dfrac{\{B_{abe}\}!}{\{B_{afc}-z\}!\{b+c-e-f+z\}!}
\dfrac{\{B_{acf}+z,B_{acf}\}}{\{z\}!}
\dfrac{\{B_{bfd}+B_{dec}-z,B_{bfd}+\alpha\}}{\{B_{dec}-\alpha-z\}!}
\Big)
\end{aligned}
\end{equation*}
\begin{equation*}
\begin{aligned}
&=
\sqrt{-1}^{B_{afc}+B_{bdf}-n_1-B_{edc}-B_{abe}}
(-1)^{b+c-e-f}
\xi_n^{-n_1(f-d)}
\qbin{2c}{\!\!A_{afc}-n_1\!\!}^{-1}
\\& 
\sum_{\alpha=0}^{B_{dec}}\!
\xi_n^{\alpha+\alpha(2d-\alpha+1)+\frac{(B_{bdf}-n_1-\alpha)(n-\alpha-B_{bfd})-(B_{edc}-\alpha)(B_{ecd}+\alpha+1)+(B_{bdf}-\alpha)(B_{bfd}+\alpha+1)}{2}}
\\
&
\dfrac{\{B_{bfd}+\alpha,B_{bfd}\}\{2d-\alpha,B_{dfb}\}\{2c,A_{edc}-n_1\}\{B_{edc}\}!^2\{B_{abe}\}!^2\{2e,A_{abe}-n_1\}}
{\{B_{bdf}\}!\, \{n_1\}!^2\,\{\alpha\}!\,\{2e,B_{ecd}+\alpha\}}
\\&
\sum_{\substack{z+w=B_{dec}-\alpha,\\ B_{dec}-B_{bdf}\leq z\leq B_{afc}}}
\!\!\!\!\!\!\!\!\!\!\!\!\!\!\!\!\!
 (-1)^{z}\xi_n^{\frac{(z-w)(B_{ecd}+\alpha+1)}{2}}
\!\!\!\!\!
\dfrac{\{B_{acf}+z,B_{acf}\}\{B_{bfd}+B_{dec}-z,B_{bfd}+\alpha\}}{\{B_{afc}-z\}!\{B_{bdf}-B_{dec}+z\}!\{z\}!\{B_{dec}-\alpha-z\}!}
\end{aligned}
\end{equation*}
\begin{equation*}
\begin{aligned}
&
=
\sqrt{-1}^{B_{afc}-B_{edc}+B_{bdf}-B_{abe}}  
 (-1)^{n_1+b+c-e-f}  
\xi_n^{\frac{B_{bdf}(B_{dfb}+1)-B_{edc}(B_{ced}+1)}{2}}  \, 
\\&
\quad
\left[\begin{matrix}
2c \\ A_{afc}-n_1
\end{matrix}
\right]^{-1}\, 
\dfrac{\{B_{edc}\}!^2 \,\{B_{abe}\}!^2\, \{2c, A_{edc}-n_1\} \, \{2e, A_{abe}-n_1\}}
 {\{B_{bdf}\}!  \, (\{n_1\}!)^2 }\,
\\[-2pt]& 
\qquad
\sum_{\alpha=0}^{B_{dec}}
\!\!\!\!
\sum_{\substack{z+w=B_{dec}-\alpha,\\ B_{dec}-B_{bdf}\leq z\leq B_{afc}}}
\!\!\!\!\!\!\!\!\!\!\!\!
(-1)^z \, \xi_n^{\frac{(z-w)(B_{ced}+\alpha+1)}{2}} \,
\dfrac{\{B_{acf}+z, B_{acf}\}\, \{B_{bfd}+\alpha+w, B_{bfd}\}}
{\{B_{afc}-z\}!\, \{B_{bdf}-\alpha-w\}!\, \{z\}!\, \{w\}!}
\\[-2pt]&\qquad\quad
\xi_n^{\frac{(B_{bdf}-\alpha)(B_{bfd}+\alpha+1)}{2}} \,
\xi_n^{-(n_1-c)\alpha+ (n_1-f)(B_{bdf}-\alpha)} \, 
\dfrac{ \{2d-\alpha, B_{dfb}\}}
{ \{2e, B_{ecd}+\alpha\}\{\alpha\}!}  
\end{aligned}
\end{equation*}
\begin{equation*}
\begin{aligned}
&
=
 (-1)^{n_1+B_{dec}} \, 
\xi_n^{-B_{dec}(B_{ced}+1)}  \, 
\dfrac{\{B_{edc}\}!^2 \,\{B_{abe}\}!^2\, \{2c, A_{edc}-n_1\} \, \{2e, A_{abe}-n_1\}}
 {\{B_{bdf}\}!  \,\{B_{afc}\}!\,  \{2c, A_{afc}-n_1\}\,  \{n_1\}! }\,
\\[-4pt]& 
\qquad
\sum_{z=\max(0, -B_{bdf}+B_{dec})}^{\min(B_{dec}, B_{afc})}
\!\!\!\!
(-1)^z \, \xi_n^{z(B_{ced}+1)} \,
\dfrac{\{B_{acf}+z, B_{acf}\}\, \{B_{bfd}+B_{dec}-z, B_{bfd}\}}
{\{B_{afc}-z\}!\, \{B_{bdf}-B_{dec}+z\}!\, \{z\}!}
\\[-2pt]&\qquad\qquad\qquad
\sum_{\alpha=0}^{B_{dec}-z}
\xi_n^{(z+2c-2n+2)\alpha} \,
\dfrac{\{2d-\alpha, B_{dfb}\}}
{\{2e, B_{ecd}+\alpha\}\,  \{\alpha\}! \, \{B_{dec}-z-\alpha\}!}\, .
\end{aligned}
\end{equation*}
Now using the equality 
$
\frac{\{2d-\alpha,B_{dfb}\}\{2d,B_{dce}+z\}}{\{2d,B_{dfb}\}}
=
\{2d-\alpha,B_{dce}+z\}
$ 
we get:
\begin{equation*}
\begin{aligned}
&
 (-1)^{n_1+B_{dec}} 
\xi_n^{-B_{dec}(B_{ced}+1)} \!
\dfrac{\{B_{edc}\}!^2 \{B_{abe}\}!^2 \{2c, A_{edc}\!\!-\!n_1\}\!  \{2e, A_{abe}\!\!-\!n_1\} \! \{2d, B_{dfb}\}}
 {\{B_{bdf}\}! \,\{B_{afc}\}!\,  \{2c, A_{afc}-n_1\} \, \{n_1\}!\, \{2e, B_{ecd}\}}\,
\\& 
\sum_{z=\max(0, -B_{bdf}+B_{dec})}^{\min(B_{dec}, B_{afc})}
\!\!\!\!\!\!\!\!
(-1)^z  \xi_n^{z(B_{ced}+1)} 
\dfrac{\{B_{acf}+z, B_{acf}\}\, \{B_{bfd}+B_{dec}-z, B_{bfd}\}}
{\{B_{afc}-z\}! \{B_{bdf}-B_{dec}+z\}! \{z\}!  
\{2d, B_{dce}+z\}}
\\&\qquad\qquad\qquad
\sum_{\alpha=0}^{B_{dec}-z}
\xi_n^{(z+2c-2n+2)\alpha} \,
\qbin{2d-\alpha}{B_{dce}+z}\, 
\qbin{B_{ecd}+\alpha}{B_{ecd}}. 
\end{aligned}
\end{equation*}
By using \eqref{eq:c}, the formula above is equal to 
\begin{equation*}
\begin{aligned}
&
 (-1)^{n-1+B_{dec}} \, 
\dfrac{\{B_{edc}\}!^2 \,\{B_{abe}\}!^2\, \{2c, A_{edc}-n_1\} \, \{2e, A_{abe}-n_1\}\, \{2d, B_{dfb}\}}
 {\{B_{bdf}\}! \,\{B_{afc}\}!\,  \{2c, A_{afc}-n_1\} \, \{n_1\}! \, \{2e, B_{ecd}\}}\,
\\& 
\sum_{z=\max(0, -B_{bdf}+B_{dec})}^{\min(B_{dec}, B_{afc})}
\!\!\!\!\!\!\!
\dfrac{(-1)^z \, \{B_{acf}+z, B_{acf}\}\, \{B_{bfd}+B_{dec}-z, B_{bfd}\}}
{\{B_{afc}-z\}!\{B_{bdf}-B_{dec}+z\}! \{z\}! 
\{2d, B_{dce}+z\}} 
\qbin{A_{edc}+1}{\!\!2c+z+1\!\!}
\end{aligned}
\end{equation*}
\begin{equation*}
\begin{aligned}
&
=
(-1)^{n-1+B_{afc}} \, 
\dfrac{\{B_{dec}\}! \, \{B_{abe}\}!}
{\{B_{bdf}\}! \, \{B_{afc}\}!} \, 
\qbin{2e}{A_{abe}-n_1} \, 
{\qbin{2e}{B_{ecd}}}^{-1} \,
\\&
\sum_{z=\max(0, -B_{bdf}+B_{dec})}^{\min(B_{dec}, B_{afc})}
(-1)^z \, 
\left[\begin{matrix}
A_{afc}+1 \\ 2c+z+1
\end{matrix}
\right] 
\left[\begin{matrix}
B_{acf}+z \\ B_{acf}
\end{matrix}
\right] 
\left[\begin{matrix}
B_{bfd}+B_{dec}-z \\ B_{bfd}
\end{matrix}
\right] 
\left[\begin{matrix}
B_{dce}+z \\ B_{dfb}
\end{matrix}
\right].  
\end{aligned}
\end{equation*}
In the last equality we used the identity 
$$
\frac{\{A_{edc}+1,2c+z+1\}\,\{2c,A_{edc}-n_1\}}{\{2c,A_{afc}-n_1\}}
=
\{A_{afc}+1,2c+z+1\}(-1)^{a+f-e-d}
$$ 
which is a direct consequence of the definition of the symbol $\{x,x-k\}$.
\qed

%
\end{document}